\documentclass{article}
\usepackage[utf8]{inputenc}
\usepackage{amsmath}
\usepackage{amssymb}
\usepackage{caption}
\usepackage{subcaption}
\usepackage{empheq}
\usepackage{amsthm,graphicx,float}
\newtheorem{theorem}{Theorem}

\title{Introducing memory to a family of multi-step multidimensional iterative methods with weight function }

\author{
{Alicia Cordero, Eva G. Villalba, Juan R. Torregrosa, Paula Triguero-Navarro}\\
{\quad}
\\
{\small
Instituto Universitario de Matem\'atica Multidisciplinar,}\\
{\small Universitat Polit\`ecnica de Val\`encia.
Val\`encia, Spain.}\\\\
{\small E-mail:
acordero@mat.upv.es, egarvil@posgrado.upv.es, jrtorre@mat.upv.es, ptrinav@doctor.upv.es,}\\}

\begin{document}

\maketitle
\begin{abstract}
In this paper, we construct a derivative-free multi-step iterative scheme based on Steffensen's method. To avoid excessively increasing the number of functional evaluations and, at the same time, to increase the order of convergence, we freeze the divided differences used from the second step and use a weight function on already evaluated operators. Therefore, we define a family of multi-step methods with convergence order $2m$, where $m$ is the number of steps, free of derivatives, with several parameters and with dynamic behaviour, in some cases, similar to Steffensen's method. In addition, we study how to increase the convergence order of the defined family by introducing memory in two different ways: using the usual divided differences and the Kurchatov divided differences. We perform some numerical experiments to see the behaviour of the proposed family and suggest different weight functions to visualize with dynamical planes in some cases the dynamical behaviour.

\end{abstract}

\section{Introduction}
One of the most studied problems in numerical analysis is the resolution of systems of nonlinear equations of the form $F(x) = 0$, where $F:D\subset\mathbb{R}^n\rightarrow \mathbb{R}^n$. The interest in this type of problems is due to the fact that it is crucial to solve other problems of greater complexity in applied mathematics. Generally, it is not possible to find the exact solution of these systems, so iterative methods are used to approximate their solutions.

Iterative methods have been an important field of research in recent years. The essence of these methods consists of finding an approximate solution, using an iterative process. More specifically, starting from an initial approximation $x^{(0)}$ close to the solution $\alpha$, a sequence of approximations $\{x^{(k)}\}_{k\geq0}$ is constructed, through the iterative method, such that, under certain conditions, $\lim\limits_{k\to\infty}x^{(k)}=\alpha$.

The most famous and widely used method for solving $F(x)=0$ is Newton's scheme (see \cite{Ostr1}), whose iteration is given by
\begin{equation}\label{New}
    x^{(k+1)}=x^{(k)}-[F'(x^{(k)})]^{-1}F(x^{(k)}),\hspace{0.2cm}k=0,1,2,\ldots
\end{equation}
where $F'(x^{(k)})$ represents the Jacobian matrix of the operator $F$ evaluated at the iteration $x^{(k)}$. The importance of this method is due to its simplicity, efficiency and the fact that, in the unidimensional case, it is an optimal method, in the sense of Kung-Traub's conjecture \cite{KT}. Moreover, another reason why this method is commonly used is that, under certain conditions, it has quadratic convergence and high accessibility, that is, the region of starting points $x^{(0)}$ for which the method converges is large. However, due to the presence of the Jacobian matrix, Newton's scheme is only applicable to differentiable operators.

For this reason, other methods that avoid the computation of the Jacobian have been developed in recent years. One of the most important derivative-free schemes is Steffensen's method \cite{St}, a method that approximates the Jacobian matrix of Newton's scheme by a first-order divided-difference operator of the form $[x,x+F(x);F]$. Other similar derivative-free schemes obtained from it are, for example, \cite{KT,Alarcon,Argyros,Ezquerro,BK,AA,He}. These methods are applicable to non-differentiable operators and, in turn, preserve the simplicity, efficiency and quadratic convergence of Newton's scheme.

Many tries have been made to improve this method from different points of view. At times, tries have been made to improve its dynamics, efficiency and convergence speed, but improving it on the one hand, in most cases, means making the iterative expression more complicated. For example, in \cite{Amat} we see that increasing its convergence speed implies performing new functional evaluations, i.e., increasing its computational cost and therefore making its efficiency worse.

In response to this type of problem, many articles focus on the construction of derivative-free multi-step iterative schemes, which manage to increase the order of convergence of methods such as Steffensen's, without excessively increasing the number of functional evaluations in each iteration. An example of this is seen in \cite{Narang}, in which a multi-step iterative process is presented based on the composition of Steffensen's method with itself $m$ times, but using in each of the $m$ steps the same divided difference operator as in the first step. With this idea it is possible to achieve a convergence order of $m+1$.

In this paper, we are inspired precisely by this idea, which is to increase the speed of convergence of Steffensen's method while trying not to deteriorate its other characteristics, such as accessibility and efficiency. For this purpose, we have constructed a parametric family of multi-step iterative methods without derivatives, making use of a weight function $H(t^{(k)})$ (see \cite{Chi,Chun}), which allow us to explore the different advantages of each of the methods that are part of this family depending on the form of this function. The family is described by the following iterative scheme:
\begin{equation}\label{Method}
SW_m: \left\{
\begin{array}{lcl}
z_1^{(k)}&=&x^{(k)}-\left[w^{(k)},x^{(k)};F\right]^{-1}F\left(x^{(k)}\right),\\
z_2^{(k)}&=&z_1^{(k)}-H\left(t^{(k)}\right)\left[w^{(k)},x^{(k)};F\right]^{-1}F\left(z_1^{(k)}\right),\\
&\vdots&\\
z_{m-1}^{(k)}&=&z_{m-2}^{(k)}-H\left(t^{(k)}\right)\left[w^{(k)},x^{(k)};F\right]^{-1}F\left(z_{m-2}^{(k)}\right),\\
x^{(k+1)}&=&z_{m-1}^{(k)}-H\left(t^{(k)}\right)\left[w^{(k)},x^{(k)};F\right]^{-1}F\left(z_{m-1}^{(k)}\right),
\end{array}
\right.
\end{equation}
where $t^{(k)}=\left[w^{(k)},x^{(k)};F\right]^{-1}\left[z_1^{(k)},v^{(k)};F\right]$, where $w^{(k)}=x^{(k)}+\beta F\left(x^{(k)}\right)$ and $v^{(k)}=z_1^{(k)}+\delta F\left(z_1^{(k)}\right)$. Note that $k$ refers to the number of iterations and $m$ to the number of steps.\\\\ Throughout this article we see the convergence speed of this family, study its efficiency and compare its accessibility with Steffensen's to make sure that the increase in speed has not affect its dynamics. Furthermore, in order to farther increase its convergence speed, we study two ways to introduce memory to our family, as \cite{ChiCo}.

With these goals in mind, we have structured the paper as follows: First, in Section 2 we analyse the convergence of the parametric family and obtain its error equation. On the other hand, in Section 3 we introduce memory to the family in two different ways and study the order of convergence of the two resulting families. In Section 4, we study the efficiency of the family depending on the size of the system to be solved and the number of steps $m$. Finally, in Section 5 several numerical experiments are carried out to study the behaviour of the different methods from several points of view, one of them the dynamical planes.

\section{Convergence analysis}
In this section, we study the local order of convergence for the family given by (\ref{Method}). For this purpose, we assume that the nonlinear operators $H$ and $F$ are differentiable, so that we can obtain the Taylor development in an environment of the identity matrix $I$ for $H$ and of the solution $\alpha$ for $F$. With this, we arrive to an explicit expression of the convergence error equation of the family and, from this expression, we deduce the order of convergence.

To this end, let us introduce the necessary notation for operator $H$ and the characterization of the divided difference operators introduced in \cite{OR}.

First, we consider $X=\mathbb{R}^{n\times n}$ the Banach space of real square matrices of size $n \times n$ and let be the differentiable function $H:X\rightarrow X$ such that its Fréchet derivatives are well defined and satisfy:
\begin{itemize}
    \item $H'(u)(v)=H_1uv$, where $H':X\rightarrow \mathcal{L}(X)$ and $H_1\in \mathbb{R}$,
    \item $H''(u,v)(w)=H_2uvw$, where $H'':X\times X\rightarrow \mathcal{L}(X)$ and $H_2\in \mathbb{R}$,
\end{itemize}
where the set of linear operators defined in X is denoted by $\mathcal{L}(X)$. Then, when the number of iterations, $k$, tends to infinity, the variable $t^{(k)}$ tends to the identity matrix $I$ and, therefore, there are real $H_1$ , $H_2$ such that $H$ can be expanded around $I$ as:
$$H\left(t^{(k)}\right)=H(I)+H_1\left(t^{(k)}-I\right)+\dfrac{1}{2}H_2\left(t^{(k)}-I\right)^2+O\left(\left(t^{(k)}-I\right)^3\right).$$
On the other hand, for the characterization of the divided difference operators, we consider that $F: \mathbb{R}^n \rightarrow \mathbb{R}^n$ is a differentiable function in an open neighborhood $D \subset \mathbb { R }^n$, which contains the solution $\alpha$. Then, we consider the divided difference operator
$$[x+h,x;F]=\int_0^1 F'(x+th)dt,$$
defined by Genochi-Hermite in \cite{OR}.

Using the Taylor expansion of $F'(x+th)$ around point $x$ and integrating, we obtain the following development
$$[x+h,x;F]=F'(x)+\dfrac{1}{2}F''(x)h+\dfrac{1}{6}F'''(x)h^2+O(h^3).$$
Then, we establish the following result.

\begin{theorem}\label{Teorema1}
Let $F: \mathbb{R} ^n \longrightarrow \mathbb{R}^n$ be a sufficiently differentiable function in an neighbourhood of $\alpha$, which we denote by $D \subset \mathbb{R}^n$, such that $F(\alpha)=0$. We assume that $F'(\alpha)$ is non singular. Let $H(t)$ be a real matrix function that satisfies  $H(0)=1$, $H_1=-1$ and $\|H''(0)\|<\infty$.
Then, taking an estimate $x^{(0)}$ close enough to $\alpha$, the sequence of iterates $\{x^{(k)}\}_{k\geq 0}$ generated by method (\ref{Method}) with $m$-steps converges to $\alpha$ with order $2m$ for all $\beta\neq0$ and $\delta\neq0$.
\end{theorem}
\begin{proof}
We are going to perform the proof by induction on the number of steps $m$. We start with the case where $m=1$.

We first obtain the Taylor expansion of $F\left(x^{(k)}\right)$ around $\alpha$, where $e_k=x^{(k)}-\alpha$:
\begin{equation*}
F\left(x^{(k)}\right)=F'(\alpha)\left(e_k+C_2e_k^2+C_3e_k^3+C_4e_k^4\right)+O\left(e_k^5\right),
\end{equation*}
where $C_j=\dfrac{1}{j}F'(\alpha)^{-1}F^{(j)}(\alpha)\in\mathcal{L}_j(X,X)$, being $\mathcal{L}_j(X,X)$ the set of $j$-linear functions, for $j=2,3,\ldots$ Then, applying derivatives to the Taylor expansion of $F\left(x^{(k)}\right)$ around $\alpha$, we obtain
\begin{align*}
F'\left(x^{(k)}\right)&=F'(\alpha)\left(I+2C_2e_k+3C_3e_k^2+4C_4e_k^3\right)+O\left(e_k^4\right),\\
F''\left(x^{(k)}\right)&=F'(\alpha)\left(2C_2+6C_3e_k+12C_4e_k^2\right)+O\left(e_k^3\right),\\
F'''\left(x^{(k)}\right)&=F'(\alpha)\left(6C_3+24C_4e_k\right)+O\left(e_k^2\right).
\end{align*}
By denoting $e_w$ as $ w^{(k)}-\alpha$, and using the definition of $w^{(k)}=x^{(k)}+\beta F(x^{(k)})$, then we obtain
\begin{align*}
    e_w=w^{(k)}-\alpha&=x^{(k)}-\alpha+\beta F\left(x^{(k)}\right)\\
    &=e_k+\beta F'(\alpha)\left(e_k+C_2e_k^2\right)+O\left(e_k^3\right)\\
    &=(I+\beta F'(\alpha))e_k+\beta F'(\alpha)C_2e_k^2+O\left(e_k^3\right).
\end{align*}
Let us now calculate $\left[w^{(k)},x^{(k)};F\right]$ using the Genochi-Hermite formula with $h=w^{(k)}-x^{(k)}=e_w-e_k$ and the last equations:
 \begin{eqnarray*}
\left[w^{(k)},x^{(k)};F\right]&=&F'\left(x^{(k)}\right)+\dfrac{1}{2}F''\left(x^{(k)}\right)h+\dfrac{1}{6}F'''\left(x^{(k)}\right)h^2+O\left(h^3\right)\\
&=&F'(\alpha)\left(I + C_2\left(2I + \beta F'(\alpha)\right)e_k +
 \left(3C_3 +  3\beta C_3 F'(\alpha) +  \beta C_2 F'(\alpha)C_2 + \beta^2 C_3 F'(\alpha)^2\right)e_k^2\right)\\
 &&+O\left(e_k^3\right).
 \end{eqnarray*}
We now obtain the inverse of the divided difference operator $
\left[w^{(k)},x^{(k)};F\right]$. It has the expression $
\left[w^{(k)},x^{(k)};F\right]^{-1}=\left(I+X_1e_k+X_2e_k^2+O\left(e_k^3\right)\right)F'(\alpha)^{-1}$, where
\begin{align*}
    X_1&=-C_2\left(2I + \beta F'(\alpha)\right),\\
    X_2&=(C_2\left(2I + \beta F'(\alpha)\right))^2-\left(3C_3 +  3\beta C_3 F'(\alpha) +  \beta C_2 F'(\alpha)C_2 + \beta^2 C_3 F'(\alpha)^2\right).
\end{align*}
Then,
\begin{align*}
    z_1^{(k)}-\alpha&=x^{(k)}-\alpha-\left[w^{(k)},x^{(k)};F\right]^{-1}F\left(x^{(k)}\right)\\
    &=e_k-\left(I+X_1e_k+X_2e_k^2+O\left(e_k^3\right)\right)\left(e_k+C_2e_k^2+O\left(e_k^3\right)\right)\\
    &=e_k-\left(e_k+C_2e_k^2+X_1e_k^2+O\left(e_k^3\right)\right)\\
    &=-(C_2+X_1)e_k^2+O\left(e_k^3\right)\\
    &=Y_2e_k^2+O\left(e_k^3\right),
\end{align*}
being $Y_2=-(C_2+X_1)$. Thus, it is proved that the $1$-step method has order 2.

For $m=2$, we calculate $\left[v^{(k)},z_1^{(k)};F\right]$, using the Genochi-Hermite formula with $h=v^{(k)}-z_1^{(k)}=-\delta F\left(z_1^{(k)}\right)$ and the last equations:
 \begin{align*}
\left[v^{(k)},z_1^{(k)};F\right]&=F'\left(z_1^{(k)}\right)+\dfrac{1}{2}F''\left(z_1^{(k)}\right)h+\dfrac{1}{6}F'''\left(z_1^{(k)}\right)h^2+O\left(h^3\right)\\
&=F'(\alpha)\left(I + C_2\left(2-\delta F'(\alpha)\right)C_2\left(I + \beta F'(\alpha)\right)e_k^2\right)+O\left(e_k^3\right).
 \end{align*}
It follows that,
\begin{align*}
    t^{(k)}&=\left[w^{(k)},x^{(k)};F\right]^{-1}\left[v^{(k)},z_1^{(k)};F\right]\\
    &=\left(I+X_1e_k+X_2e_k^2+O\left(e_k^3\right)\right)\left(I + C_2\left(2-\delta F'(\alpha)\right)C_2\left(I + \beta F'(\alpha)\right)e_k^2\right)+O\left(e_k^3\right)\\
    &=I+X_1e_k+(X_2+C_2\left(2-\delta F'(\alpha)\right)C_2\left(I+ \beta F'(\alpha)\right))e_k^2+O\left(e_k^3\right)\\
    &=I+X_1e_k+(X_2+C_2\left(2-\delta F'(\alpha)\right)C_2\left(I+ \beta F'(\alpha)\right))e_k^2+O\left(e_k^3\right),
\end{align*}
where $T_2=X_2+C_2\left(2-\delta F'(\alpha)\right)C_2\left(I+ \beta F'(\alpha)\right)$.

It follows that
\begin{align*}
    H\left(t^{(k)}\right)&=H(I)+H_1\left(t^{(k)}-I\right)+\dfrac{1}{2}H_2\left(t^{(k)}-I\right)^2+O\left(\left(t^{(k)}-I\right)^3\right)\\
    &=I-\left(t^{(k)}-I\right)+\dfrac{1}{2}H_2\left(t^{(k)}-I\right)^2+O\left(\left(t^{(k)}-I\right)^3\right)\\
    &=I-\left(X_1e_k+T_2e_k^2\right)+\dfrac{1}{2}H_2X_1^2e_k^2+O\left(e_k^3\right)\\
    &=I-X_1e_k+\left(\dfrac{1}{2}H_2X_1^2-T_2\right)e_k^2+O\left(e_k^3\right).
\end{align*}
Thus,
\begin{equation}\label{eq:Hdif}
\begin{split}
   H\left(t^{(k)}\right)\left[w^{(k)},x^{(k)};F\right]^{-1}&=\left(I-X_1e_k+\left(\dfrac{1}{2}H_2X_1^2-T_2\right)e_k^2\right)\left(I+X_1e_k+X_2e_k^2\right)F'(\alpha)^{-1}+O\left(e_k^3\right)\\
    &= \left(I-\left(\left(\dfrac{1}{2}H_2-1\right)X_1^2-T_2+X_2\right)e_k^2\right)F'(\alpha)^{-1}+O\left(e_k^3\right)\\
    &=\left(I-G_2e_k^2\right)F'(\alpha)^{-1}+O\left(e_k^3\right),
    \end{split}
\end{equation}
where $G_2=\left(\dfrac{1}{2}H_2-1\right)X_1^2-T_2+X_2$. So,
\begin{align*}
  z_2^{(k)}-\alpha&=z_1^{(k)}-\alpha -H\left(t^{(k)}\right)\left[w^{(k)},x^{(k)};F\right]^{-1}F\left(z_1^{(k)}\right).
    \end{align*}
    Denoting $e_{z_1}=z_1^{(k)}-\alpha$, then
    \begin{align*}
  z_2^{(k)}-\alpha &=z_1^{(k)}-\alpha-\left(I-G_2e_k^2+O\left(e_k^3\right)\right)\left(e_{z_1}+C_2e_{z_1}^2+O\left(e_{z_1}^3\right)\right)\\
   &=e_{z_1}-\left(e_{z_1}-G_2e_k^2e_{z_1}+C_2e_{z_1}^2\right)+O\left(e_k^5\right)\\
   &=G_2e_k^2e_{z_1}-C_2e_{z_1}^2+O\left(e_k^5\right)\\
   &=(G_2Y_2-C_2Y_2^2)e_k^4+O\left(e_k^5\right)\\
   &=(G_2-C_2Y_2)Y_2e_k^4+O\left(e_k^5\right).
\end{align*}
Thus, it is proved that the two-step method has order $4$.

We assume, by induction on the number of steps, that for $m-1$ the error equation is
\begin{align*}
e_{z_{m-1}}&=G_2^{m-3}\left(G_2-C_2Y_2\right)Y_2e_k^{2(m-1)}+O\left(e_k^{2m-1}\right),
\end{align*}
being $e_{z_{m-1}}=z_{m-1}^{(k)}-\alpha$. Then,
\begin{align*}
e_{k+1}&=e_{z_{m-1}}-H\left(t^{(k)}\right)\left[w^{(k)},x^{(k)};F\right]^{-1}F\left(z^{(k)}_{m-1}\right)\\
&=e_{z_{m-1}}-\left(I-G_2e_k^2+O\left(e_k^3\right)\right)\left(e_{z_{m-1}}+C_2e_{z_{m-1}}^2+O\left(e_{z_{m-1}}^3\right)\right)\\
&=G_2e_k^2e_{z_{m-1}}+O\left(e_k^{2m+1}\right)\\
&=G_2^{m-2}(G_2-C_2Y_2)Y_2e_k^{2m}+O\left(e_k^{2m+1}\right).
\end{align*}
Thus, it is proved that the order of convergence of the $m$-step method of the family (\ref{Method}) is $2m$.
\end{proof}
In particular, by the expression of $G_2$, if $H_2=2$, we have
$$G_2=-C_2(2I-\delta F'(\alpha))C_2(I+\beta F'(\alpha)),$$
and also
$$(G_2-C_2Y_2)Y_2=-C_2\left(1-\delta F'(\alpha)\right)C_2\left(I+ \beta F'(\alpha)\right)C_2(I+\beta F'(\alpha)).$$
Thus, it is obtained that, for $m\geq 2$,
\begin{equation}\label{eq:Error3}
e_{k+1}\sim (I-\delta F'(\alpha)) \left(2I-\delta F'(\alpha)\right)^{m-2}\left(I+\beta F'(\alpha)\right)^me_k^{2m}.
\end{equation}

\section{Introducing memory}
As mentioned in the Introduction, one way to obtain iterative methods is to modify known methods with the intention of improving them in some aspect, either by eliminating the derivatives used in the method or by increasing the order.

One way to try to increase the order of convergence is to introduce memory to the iterative method. When we introduce memory what we do is to use the previous iterates and the functional evaluations that we have already calculated, in order to increase the order without carrying out new evaluations and, therefore, make the methods more efficient and optimal.

In this case, we replace parameters used by our methods with an expression that combines functional evaluations of previous iterates and the iterates themselves.

The error equation, if $H_2=2$ and $m\geq 2$, is
\begin{equation}\label{eq:Error}
e_{k+1}\sim (I-\delta F'(\alpha)) (2I-\delta F'(\alpha))^{m-2}(I+\beta F'(\alpha))^me_k^{2m}.
\end{equation}
If it should happen that $\beta=-F'(\alpha)^{-1}$ or $\delta=2F'(\alpha)^{-1}$, then the order of the $m$-step method would be at least $2m+1$. But we do not know $\alpha$ or $F'(\alpha)$, so we cannot substitute parameters in this way. What we can do is approximate $F'(\alpha)$ by a combination of functional evaluations of the iterates in $F$. We do this with previous iterates and not with current iterates so as not to increase the number of functional evaluations.

One known way to approximate $F'(\alpha)$ is by the divided difference operators. In this case, we use $[x^{(k)},x^{(k-1)};F]$, so an approximation of parameters will be as follows
$$\beta_k=- [x^{(k)},x^{(k-1)};F]^{-1}\hspace{0.2cm}\text{ and }\hspace{0.2cm} \delta_k=2[x^{(k)},x^{(k-1)};F]^{-1}.$$
If we replace parameters $\beta$ and $\delta$ of the iterative family (\ref{Method}) by the previous approximations, we obtain a new family of iterative methods, which we denote by $SWD_m$. We check now that the order of convergence of this family is $m+\sqrt{m^2+2m-2}$, for $m\geq 2$.

The divided difference operator defined by Kurchatov, which has the expression, $[2x-y,y;F]$ obtains good approximations, so in this case we also use it to approximate $F'(\alpha)$. The second way to approximate parameters is as follows
$$\beta_k=- [2x^{(k)}-x^{(k-1)},x^{(k-1)};F]^{-1}\hspace{0.2cm}\text{ and }\hspace{0.2cm} \delta_k=2[2x^{(k)}-x^{(k-1)},x^{(k-1)};F]^{-1}.$$
If we replace parameters $\beta$ and $\delta$ of the iterative family (\ref{Method}) by the previous approximations, we obtain a new family of iterative methods, which we denote by $SWK_m$. We check that the order of convergence of this family is $m+\sqrt{m^2+4m-4}$, for $m\geq 2$.

%%%%%%%%%%%%%%%%%%%%%%%%%%%%%%%%%%%%%%%%%%%%%%%%%%%%%%%%%%%%%%%%%%%%%%%%%%%%%%%%%%%%%%%%
\begin{theorem}\label{TeoremaM1}
Let $F: \mathbb{R} ^n \longrightarrow \mathbb{R}^n$ be a sufficiently differentiable function in an neighbourhood of $\alpha$, which we denote by $D \subset \mathbb{R}^n$, such that $F(\alpha)=0$. We assume that $F'(\alpha)$ is non singular. Let $H(t)$ be a real matrix function that verifies that $H(0)=1$, $H_1=-1$ and $H_2=2$. Then, taking an estimate $x^{(0)}$ close enough to $\alpha$, the sequence of iterates $\{x^{(k)}\}_{k\geq 0}$ generated by the $SWD_m$ method converges to $\alpha$ with order $m+\sqrt{m^2+2m-2}$ and the sequence of iterates $\{x^{(k)}\}_{k\geq 0}$ generated by the $SWK_m$ method converges to $\alpha$ with order of convergence $m+\sqrt{m^2+4m-4}$, for $m\geq 2$.
\end{theorem}
\begin{proof}
Let us now consider the Taylor expansion of $F\left(x^{(k-1)}\right)$, $F'\left(x^{(k-1)}\right)$ and $F''\left(x^{(k-1)}\right)$ around $\alpha$ in the same way as considered in Theorem \ref{Teorema1}.
On the one hand, let us start calculating $\left[x^{(k)},x^{(k-1)};F\right]$, using the Genochi-Hermite formula with $h=e_k-e_{k-1}$, thus
\begin{align*}
 \left[x^{(k)},x^{(k-1)};F\right]&=F'(\alpha)\left(I+C_2(e_k+e_{k-1})\right)+O_2\left(e_k,e_{k-1}\right),
\end{align*}
where $O_2(e_k,e_{k-1})$ denotes all terms in
where the sum of exponents of $e_k$ and $e_{k-1}$ is at least $2$.
Then, the inverse of this divided difference operator is:
\begin{align*}
 \left[x^{(k)},x^{(k-1)};F\right]^{-1}&=(I-C_2(e_k+e_{k-1}))F'(\alpha)^{-1}+O_2\left(e_k,e_{k-1}\right).
\end{align*}
Therefore,
$\beta_k=-(I-C_2(e_k+e_{k-1}))F'(\alpha)^{-1}+O_2\left(e_k,e_{k-1}\right)$ and
 \begin{align*}
 I+\beta_kF'(\alpha)
 &=C_2(e_k+e_{k-1}))+O_2(e_{k-1},e_k).
 \end{align*}
 Furthermore, $\delta_k=2(I-C_2(e_k+e_{k-1}))F'(\alpha)^{-1}+O_2\left(e_k,e_{k-1}\right)$ and
 \begin{align*}
 2I-\delta_kF'(\alpha)
 &=2C_2(e_k+e_{k-1}))+O_2(e_{k-1},e_k).
 \end{align*}
 Thus, $I+\beta_kF'(\alpha)\sim e_{k-1}$ and  $2I-\delta_kF'(\alpha)\sim e_{k-1}$.

By the error equation (\ref{eq:Error}) and the above relation we have
\begin{equation}\label{eq:error:memoria}
e_{k+1}\sim e_{k-1}^{m-2}e_{k-1}^me_k^{2m}\sim e_{k-1}^{2m-2}e_{k}^{2m}.
\end{equation}
On the other hand, suppose that the R-order of the method is at least $p$. Therefore, it is satisfied
$$e_{k+1}\sim D_{k,p}e_k^p,$$
where $D_{k,p}$ tends to the asymptotic error constant, $D_p$, when $k\xrightarrow{}\infty$.
Then, one has that
\begin{equation}\label{igual1_orden_memk1}
    e_{k+1}\sim D_{k,p}\left(D_{k-1,p}e_{k-1}^p\right)^p=D_{k,p}D_{k-1,p}^pe_{k-1}^{p^2}.
\end{equation}
In the same way that the relation (\ref{eq:error:memoria}) is obtained, it follows that
\begin{equation}\label{igual2_orden_memK1}
    e_{k+1}\sim e_{k-1}^{2m-2}\left(D_{k-1,p}e_{k-1}^p\right)^{2m}=D_{k-1,p}^{2m}e_{k-1}^{2m-2+2mp}.
\end{equation}
Then by equaling the exponents of $e_{k-1}$ in (\ref{igual1_orden_memk1}) and (\ref{igual2_orden_memK1}), we obtain
$$p^2=2mp+2m-2,$$
whose only positive solution is the order of convergence of the $SWD_m$ method, being $p=m+\sqrt{m^2+2m-2}$ for $m\geq 2$.

%HASTA AQUÍ
On the other hand, applying Genochi-Hermite formula on Kurchatov divided differences, we obtain
\begin{align*}
 \left[2x^{(k)}-x^{(k-1)},x^{(k-1)};F\right]&=F'(\alpha)\left(I+2C_2e_k-2C_3e_{k-1}e_k+C_3e_{k-1}^2+4C_3e_k^2\right)+O_3\left(e_k,e_{k-1}\right).
\end{align*}
Then, the inverse of this divided difference operator is:
\begin{align*}
 \left[2x^{(k)}-x^{(k-1)},x^{(k-1)};F\right]^{-1}&=\left(I-2C_2e_k-C_3e_{k-1}^2+2C_3e_{k-1}e_k+4(C_2^2-C_3)e_k^2\right)F'(\alpha)^{-1}\\&+O_3\left(e_k,e_{k-1}\right).
\end{align*}
Therefore,
 \begin{align*}
 I+\beta_kF'(\alpha)&=2C_2e_k+C_3e_{k-1}^2-2C_3e_{k-1}e_k-4\left(C_2^2-C_3\right)e_k^2+O_3\left(e_k,e_{k-1}\right),\\
 2I-\delta_kF'(\alpha)&=4C_2e_k+2C_3e_{k-1}^2-4C_3e_{k-1}e_k-8\left(C_2^2-C_3\right)e_k^2+O_3\left(e_k,e_{k-1}\right).
 \end{align*}
 Thus, $I+\beta_kF'(\alpha)$ and $2I-\delta_k F'(\alpha)$ can have the behaviour of $e_k$, $e_ke_{k-1}$, $e_{k}^2$ or $e_{k-1}^2$.
 Obviously the factors $e_ke_{k-1}$ and $e_{k}^2$ tend faster to 0 than $e_k$, so we have to see which of the factors converge faster, $e_k$ or $e_{k-1}^2$.

 Assume that the R-order of the method is at least $p$. Therefore, it is satisfied
$$e_{k+1}\sim {D_{k,p}e_{k}^p},$$
where $D_{k,p}$ tends to the asymptotic error constant, $D_p$, when $k\xrightarrow{}\infty$.
  Then we have that $$\dfrac{e_k}{e_{k-1}^2}\sim \dfrac{D_{k-1,p}e_{k-1}^p}{e_{k-1}^2}.$$
 Then, if $p>2$, we obtain that $\dfrac{D_{k-1,p}e_{k-1}^p}{e_{k-1}^2}$ converges to $0$ when $k\xrightarrow{}\infty$. Thus, if $p>2$, then $I+\beta_kF'(\alpha)\sim e_{k-1}^2$ and $2I-\delta_k F'(\alpha) \sim e_{k-1}^2$.

 From the error equation (\ref{eq:Error}) and the above relation we obtain
\begin{equation}\label{eq:error:memoria2}
e_{k+1}\sim  \left(e_{k-1}^2\right)^{m-2}\left(e_{k-1}^2\right)^me_k^{2m}\sim e_{k-1}^{4m-4}e_k^{2m}.
\end{equation}
On the other hand, by assuming that the R-order of the method is at least $p$ we have the relation (\ref{igual1_orden_memk1}). In the same way that we obtain the relation (\ref{eq:error:memoria2}), we obtain
\begin{equation}\label{igual2_orden_mem2}
    e_{k+1}\sim e_{k-1}^{4m-4}\left(D_{k-1,p}e_{k-1}^p\right)^{2m}=D_{k-1,p}^2e_{k-1}^{2mp+4m-4}.
\end{equation}
Then, equaling the exponents of  $e_{k-1}$ in (\ref{igual1_orden_memk1}) and (\ref{igual2_orden_mem2}), it follows that
$$p^2=2mp+4m-4,$$
whose only positive solution is the order of convergence of the $SWK_m$ method, being $p=m+\sqrt{m^2+4m-4}$, $m\geq 2$.
\end{proof}

\section{Efficiency study}
One of the dilemmas we encounter when designing or improving iterative methods for solving systems of non-linear equations is whether the increase in the speed of convergence with respect to other methods is "worth it" from the point of view of computational cost. To solve these dilemmas, we make use of the concept of efficiency, which is often measured using the efficiency index, defined in \cite{traub64}:
$$I=p^{\dfrac{1}{d}},$$
where $p$ is the order of convergence and $d$ is the number of functional evaluations per iteration.

This criterion for comparing methods is very useful as it establishes a relationship between the order of convergence of a method and the number of functional evaluations it performs per iteration.

Typically, in many publications, use is made of the definition of efficiency index for the multidimensional case, taking into account that, for a system of size $n \times n$, $n$ functional evaluations are required for a vector function $F$, $n^2$ functional evaluations for a Jacobian matrix $J_F$ and $n^2-n$ functional evaluations for a first order divided-difference operator of the form $[x,y; F]$ (see \cite{Potra}) with $[x,y;F]_{ij}$ defined for all $i,j=1,\ldots,n$ as:
$$[x,y;F]_{ij}=\dfrac{1}{x_j-y_j}[F_i(x_1,\ldots,x_j,y_{j+1},\ldots,y_p)-F_i(x_1,\ldots,x_{j-1},y_j,\ldots,y_p)].$$

%In this article, we make use of these concepts, also assuming that we are solving a system of size $n\times n$, to calculate the efficiency index of the memoryless iterative family (\ref{Method}) as a function of the number of steps $m$ it performs.
\begin{itemize}
    \item For $m=1$, iterative family (\ref{Method}) performs an evaluation of $F$ and calculate a divided difference operator, so the number of functional evaluations is:
    \[n+(n^2-n)=n^2.\]
    Therefore, the efficiency ratio is:
    \[2^{^{\dfrac{1}{n^2}}}.\]
    \item For $m=2$ it needs two functional evaluations of $F$ and two divided difference operators, so the number of functional evaluations is:
    \[2n+2(n^2-n)=2n^2.\]
    So, the efficiency ratio is:
    \[4^{^{\dfrac{1}{2n^2}}}=2^{^{\dfrac{1}{n^2}}}.\]
    \item For $m>2$ it performs $m$ functional evaluations of $F$ and two divided difference operators, so the number of functional evaluations is:
    \[ m n+2(n^2-n)=2n^2+(m-2) n.\]
    The resulting efficiency ratio is
    \[(2m)^{\dfrac{1}{2n^2+(m-2) n}}.\]
\end{itemize}
That is, we have two cases:
when $m=1$ we have
 \[I_1=2^{^{\dfrac{1}{n^2}}},\]
  and when $m>1$ we have
 \[I_m=(2m)^{\dfrac{1}{2n^2+(m-2)n}}.\]
  Since $I_1=I_2$, then to obtain the maximum efficiency rate we only have to study the maximum efficiency rate for $m\geq2$.

  We want to see when the maximum efficiency rate is obtained as a function of the size of the system, that is, as a function of $n$.

  It is not difficult to see that if $n=1$, then $I_1=I_2=2$ and $I_m=(2m)^{\dfrac{1}{m}}$.
It happens that $I_m<2$ for all $m>2$ since $2m<2^m$.
 Thus, being $n=1$, the highest efficiency index will be $I_1=I_2=2$.

 We assume from here on that $n>1$.
 To obtain the maximum we are going to derive the function $I_m$. We obtain that the derivative is
 \[I_m'=I_m\dfrac{2n^2+(m-2)n-mn\ln(2m)}{m(2n^2+(m-2)n)^2}.\]
Since $I_m$ is different from $0$, then a critical point is reached when
 $2n^2+(m-2)n-mn\ln(2m)=0$, that is, $2n+m-2-m\ln(2m)=0$,
 which is equivalent to solving
 \begin{equation}\label{mast}
     (1-\ln(2m))m=2-2n.
 \end{equation}
 We denote by $m^*$ the solution of the previous equation. Then, a critical point of $I_m$ is obtained at $m^*$.

 We now observe the intervals of growth and decay of the function $I_m$. Since $I_m>0$ and $m(2n^2+(m-2)n)^2>0$, we have to study when $2n+m-2-m\ln(2m)$ is negative or positive.

 If $m=2$, then $m=2$, it follows that $2n+2-2-2\ln(4)=2n-2\ln(4)\approx 2n-2\cdot1.3863= 2n-2\cdot1.3863>0$, whenever $n$ is strictly greater than 1. Thus, we have that the function $I_m$ grows up to $m^*$, provided that $n>1$.

 Now we prove that, starting from $m^*$, we obtain that the function decreases. To do this, we take the point $m^*+\varepsilon$ with $\varepsilon>0$ and prove that $I'_{m^*+\varepsilon}<0$, since this will imply that the efficiency index $I_m$ decreases from $m^*$.

 Using the expression for $I'_m$, we get
 \begin{align*}
I'_{m^*+\varepsilon}<0\Longleftrightarrow 2n-2+(1-ln(2(m^*+\varepsilon))(m^*+\varepsilon)<0.
 \end{align*}
 Since we have that $m^*$ verifies the equation (\ref{mast}), then
  \begin{align}\label{desig}
  \begin{split}
I'_{m^*+\varepsilon}<0&\Longleftrightarrow m^*\left(\ln(2m^*)-\ln(2(m^*+\varepsilon))\right)+\varepsilon\left(1-\ln(2(m^*+\varepsilon))\right)<0\\
&\Longleftrightarrow m^*\ln\left(\dfrac{m^*}{m^*+\varepsilon}\right)+\varepsilon\left(1-\ln(2(m^*+\varepsilon))\right)<0\\
&\Longleftrightarrow m^*\ln\left(\dfrac{m^*}{m^*+\varepsilon}\right)<\varepsilon\left(\ln(2(m^*+\varepsilon))-1\right).
\end{split}
 \end{align}
 We know that, since $\varepsilon>0$ and $m^*\geq 2$, the left-hand side of the above inequality is negative. On the other hand, we deduce that the following is satisfied
 $$\ln(2(m^*+\varepsilon))-1>0\Longleftrightarrow 2(m^*+\varepsilon)>e,$$
 and this will always be true because $m^*\geq 2$. Therefore, the right-hand side of (\ref{desig}) is positive, which implies that $I'_{m^*+\varepsilon}<0$.

 Since we have proved that $I'_2>0$, that $I'_{m^*}=0$ and that $I'_{m^*+\varepsilon}<0$, we have that in $m^*$ there is a relative maximum of the function $I_m$.

 It may happen that the solution $m^*$ is not an integer, so if $m^*$ is the solution and $\bar{m}$ is the integer of $m^*$, we calculate in each case whether $I_{\bar{m}}$ is greater or less than $I_{\bar{m}+1}$.

We denote by $M$ the value that verifies to be the maximum of $I_m$ with $m>2$. We want to see if $I_1$ is greater than this maximum or not, but as $I_1=I_2$, then $I_M>I_1$ whenever $M>2$.

 Thus, the number of steps that obtains a higher efficiency rate depending on the size of the system to be solved is analysed.

In Figure \ref{fig:my_label} we see the size of the system in relation to the number of steps for obtaining the best efficiency index. In Figure \ref{fig:my_labe2}, we show for a fixed size of the system, $n=100$, which is the relation between the number of steps and the efficiency index.
 \begin{figure}[H]
     \centering
     \includegraphics[width=\textwidth]{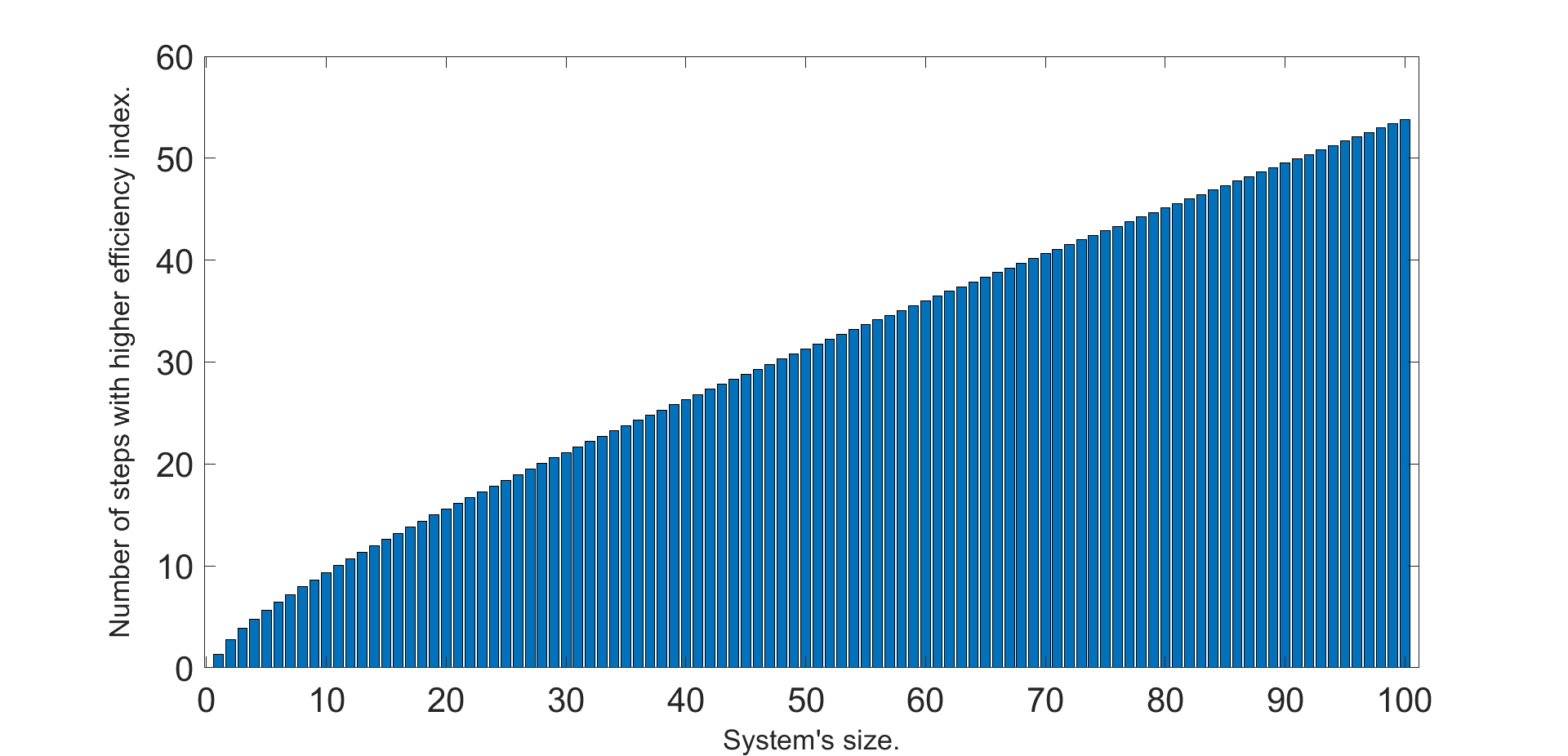}
     \caption{The number of steps that obtains a higher efficiency rate depending on the size of the system to be solved.}
     \label{fig:my_label}
 \end{figure}

 \begin{figure}[H]
     \centering
     \includegraphics[width=\textwidth]{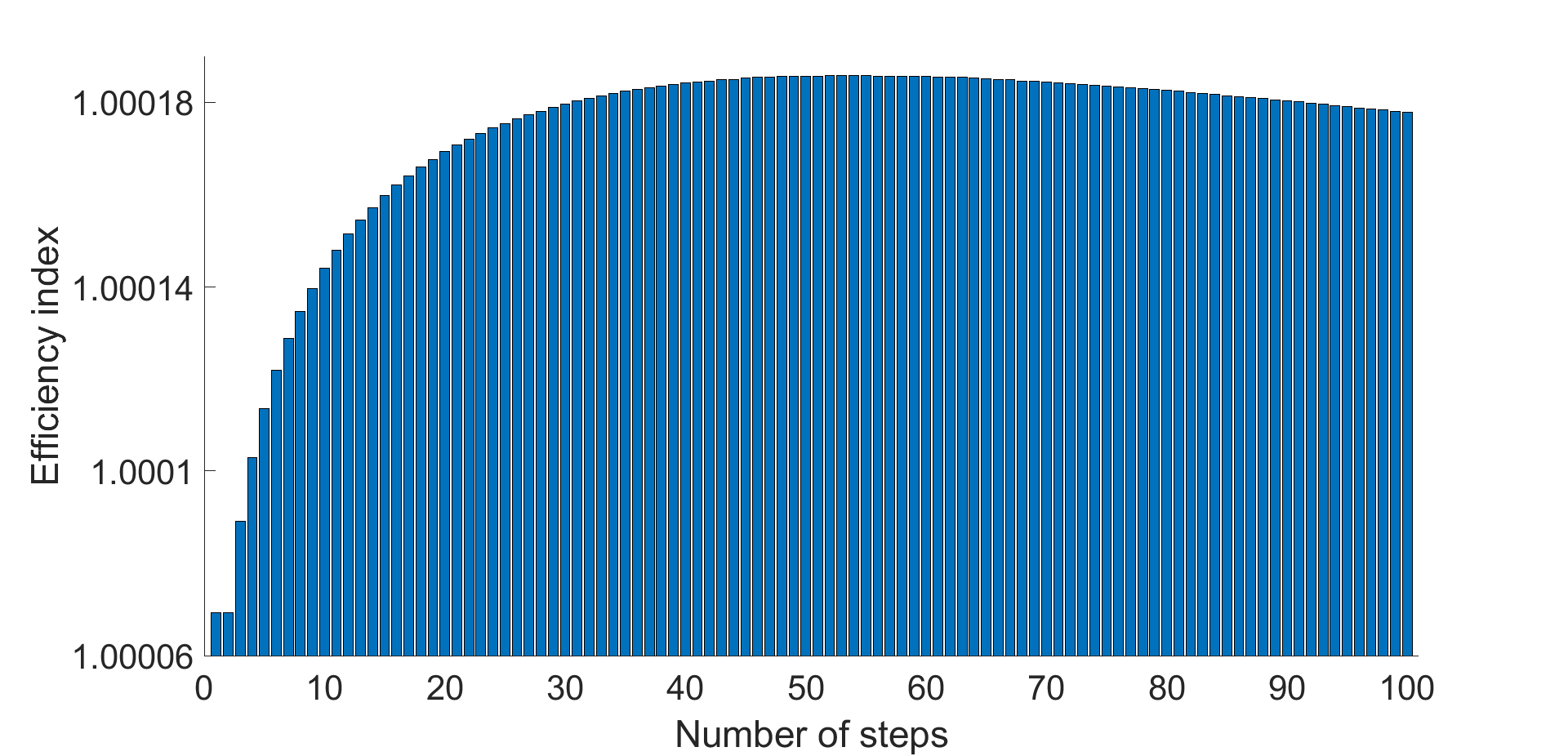}
     \caption{Efficiency rates when the system size is 100.}
     \label{fig:my_labe2}
 \end{figure}

\section{Numerical experiments}
In this section, first, we apply the multistep methods $SW_m$, $SWD_m$ and $SWK_m$ to a nonlinear system to verify that the properties deduced theoretically in the analysis of the family are satisfied, both with and without memory. Moreover, in the last section, we show several dynamical planes for our family of methods with different values of $m$ and different expressions of the weight function $H(t^{(k)})$ to verify that some of the methods of this family, with higher order of convergence than Steffensen, improve or, at least, preserve the dynamics of this method.

We want to approximate the solution of the following nonlinear system with $n$ equations and $n$ unknowns
\begin{align*}
    &F_i(x)=x_i\sin(x_{i+1})=1,\qquad i\in\{1,\ldots,n-1\}\\
    &F_n(x)=x_n\sin(x_{1})=1.
\end{align*}
The approximate solution of this system is $\alpha\approx [1.11415714087193,\ldots,1.11415714087193,\ldots]^T,$ which we try to approximate using the methods $SW_m$, $SWD_m$ and $SWK_m$, for different values of $m$.

For the computational calculations we use MATLAB R2022a, using variable precision arithmetic with 5000 digits, iterating from an initial estimate $x^{(0)}=\left[1. 3,\ldots,1.3\right]^T$ until the following stopping criterion is satisfied: $$|x^{(k+1)}-x^{(k)}\|_2+\|F(x^{(k+1)})\|_2<10^{-300}$$
and the approximated computational order of convergence (ACOC), defined by Cordero and Torregrosa in \cite{Acoc}, which has the following expression:
\[
p \approx ACOC =\frac{\ln\left(\left\|x^{(k+1)}-x^{(k)} \right\|_2/\left\|x^{(k)}-x^{(k-1)}\right\|_2\right)}{\ln\left(\left\|x^{(k)}-x^{(k-1)}\right\|_2/\left\|x^{(k-1)}-x^{(k-2)}\right\|_2\right)}.
\]

Table \ref{Tabla} shows the results obtained by the above methods to solve the system, taking $n=15$ and assuming that the weight function of the family of methods has the expression
$$H(t^{(k)})=I_n-(t^{(k)}-I_n)+(t^{(k)}-I_n)(t^{(k)}-I_n).$$ It is interesting to note that we took such a small system size because, testing with larger system sizes, such as n=100, we obtain the same numerical results, with the only exception that an exponential growth of the computation time is observed.

In addition, the value of parameters is $\beta=\delta=0.1$ for the $SW_m$ method and for the methods with memory $SWD_m$ and $SWK_m$ we use as initial aproximation $x^{(-1)}$, vector $\left[1.4,\ldots,1.4\right]^T$.

The data we compare in Table \ref{Tabla} symbolize, from left to right, the multistep methods used for different steps $m=1,2,3$, the distance between the last two iterations, the value of the function evaluated in the last iteration, the number of iterations needed to verify the stopping criterion, the approximate computational convergence order defined in \cite{Acoc} and the time it takes for each method to find an approximation to $\alpha$, satisfying the required tolerance.
\vspace{0.1cm}
\newline
\begin{table}[H]
\centering
\begin{tabular}{lccccc}
\hline
Method  & $\quad\|x^{(k+1)}-x^{(k)}\|$& $\quad\|F(x^{(k+1)})\|$ & Iteration &ACOC&Time\\\hline
$SW_1$  & $6.32894\times 10^{-439}$& $6.95714\times 10^{-879}$& $9$&$1.99999$ &$43.1406$\\
$SWD_1$  &$9.90630\times 10^{-393}$ & $4.98953\times 10^{-3553}$& $7$&$2.41381$ &$66.7344$\\
$SWK_1$  & $4.60983\times 10^{-643}$& $3.99914\times 10^{-4197}$& $7$& $2.72948$&$62.6406$\\ \hline
$SW_2$  & $2.83143\times 10^{-508}$& $1.74600\times 10^{-2036}$& $5$&$3.99999$ & $55.0469$\\
$SWD_2$  &$6.06526\times 10^{-806}$ &$4.98953\times 10^{-3553}$ & $5$& $4.44483$& $73.7188$\\
$SWK_2$  & $5.74472\times 10^{-984}$& $3.99914\times 10^{-4197}$& $5 $&$4.81006$ & $73.9063$\\ \hline
$SW_3$  & $7.33069\times 10^{-408}$& $1.23838\times 10^{-2452}$& $4$&$5.99999$ & $42.7188$\\
$SWD_3$  & $8.36495\times 10^{-549}$&$4.98953\times 10^{-3553}$ &$4$ &$6.49766$ &$60.1875$\\
$SWK_3$ & $1.19588\times 10^{-610}$& $3.99914\times 10^{-4197}$& $4$& $6.95979$&$57.7813$\\
\hline
\end{tabular}
\caption{Numerical results for $SW_m$, $SWD_m$ and $SWK_m$.}	
\label{Tabla}
\end{table}
We are going to analyse the different methods according to the number of steps $m$. In the case $m=1$, the $SW_1$ method is a modification of Steffensen's method due to the parameter $\beta=0.1$. Taking this fact into account, we observe that the methods with memory approximate the solution $\alpha$ of the problem with a higher convergence speed and in a lower number of iterations. Moreover, the approximation finally obtained is closer to the solution in these cases because the approximation errors are smaller. However, we see that both methods with memory take longer to find the solution of the system because more functional evaluations are performed per iteration, that is, the computational cost increases and, therefore, they take longer than $SW_1$.

On the other hand, if we compare the $SWD_1$ and $SWK_1$ methods, we observe that the method with memory using Kurchatov's divided differences is better than the other method with memory in all aspects. This fact is perceived in that $SWK_1$ approximates the exact solution better by the results of the second and third column, performs the same number of iterations as $SWD_1$, has a higher $ACOC$ and, in addition, takes less time to approximate the solution $\alpha$.

In the case where the methods have $m=2$ steps, that is, when the weight function is already relevant, we observe improvements in almost all aspects with respect to the $m=1$ case. We see that, although the execution times increase a little, we obtain better $ACOC$ in the three cases and, in addition, we perform a smaller number of iterations. If we compare the three methods, we obtain similar conclusions to those obtained in the case $m=1$. The $SWK_2$ method is better in almost all aspects, with the exception of execution time.

Finally, for $m=3$, we obtain the best results. In this case, thanks to the decrease in the number of iterations, we get even shorter execution times than in the $m=1$ case. Moreover, as expected from the convergence theorems, the methods show very high convergence orders. It is important to note that the $SWK_3$ method has increased the $ACOC$ of $SW_3$ by almost one unit.
%------------------------------------------------------
 \subsection{Example of dynamical planes}
In the following, as part of the numerical experiments, we show the dynamical behaviour of various members of the family of methods studied, using dynamical planes, and in some cases we also analyse the behaviour of the designed iterative methods that have memory. We apply it both to non-linear equations and to systems of equations.

We begin by commenting on the problems we are going to solve:
 \begin{itemize}
     \item The nonlinear equation $p_1:(x-1)^3-1=0$. We now that the roots of this nonlinear equation are: $2$, $\frac{1+\sqrt{3}\text{i}}{2}$ and $\frac{1-\sqrt{3}\text{i}}{2}$.
     \item The system of nonlinear equations, denoted by $p_2$, of which we analyse the behaviour is:
\begin{empheq}[left = \empheqlbrace]{align*}
               x^2-1&= 0,\\
               y^2-1 &= 0,
\end{empheq}
where $(x,y)^T\in \mathbb{R}^2$. We know that the roots of this system are: $(-1,-1)^T$, $(-1, 1)^T$, $(1,-1)^T$ and $(1,1)^T$.
 \end{itemize}
We specify in each case which weight function we are using. Remember that the methods without memory have two parameters, which we set $\delta=\beta=0.1$ to make the dynamical planes.

We start by generating the dynamical planes for the non-linear equation $p_1$. In the case of the methods without memory we generate the dynamical planes as follows.

We denote by $z$ a complex initial estimate. Each point $z$ in the plane is considered as the initial point of the iterative method. In one axis we have $\text{Re}(z)$ and in the other axis $\text{Im}(z)$. Each point $z$ is painted in a different colour depending on the point to which it converges, where is determined that the inicial point converges to one of the solutions if the distance of the iterations to that solution is less than $10^{-3}$.

These dynamical planes have been generated with a grid of $400\times400$ points and a maximum of $80$ iterations per point. We paint in orange the initial points that converge to the root $2$, in green the initial points that converge to the root $\frac{1-\sqrt{3}\text{i}}{2}$ and in purple the initial points that converge to the root $\frac{1+\sqrt{3}\text{i}}{2}$.
 \begin{figure}[H]
   \caption{Dynamical planes of $SW_1$ with different weight functions for $p_1$}
     \label{fig1}
     \centering
     \begin{subfigure}{0.45\textwidth}
     \centering
     \caption{Dynamical plane of $SW_1$ for $p_1$ with $H(t)=1-(t-1)+(t-1)^2$}
     \label{fig:Ste2}
     \includegraphics[width=1.5\textwidth]{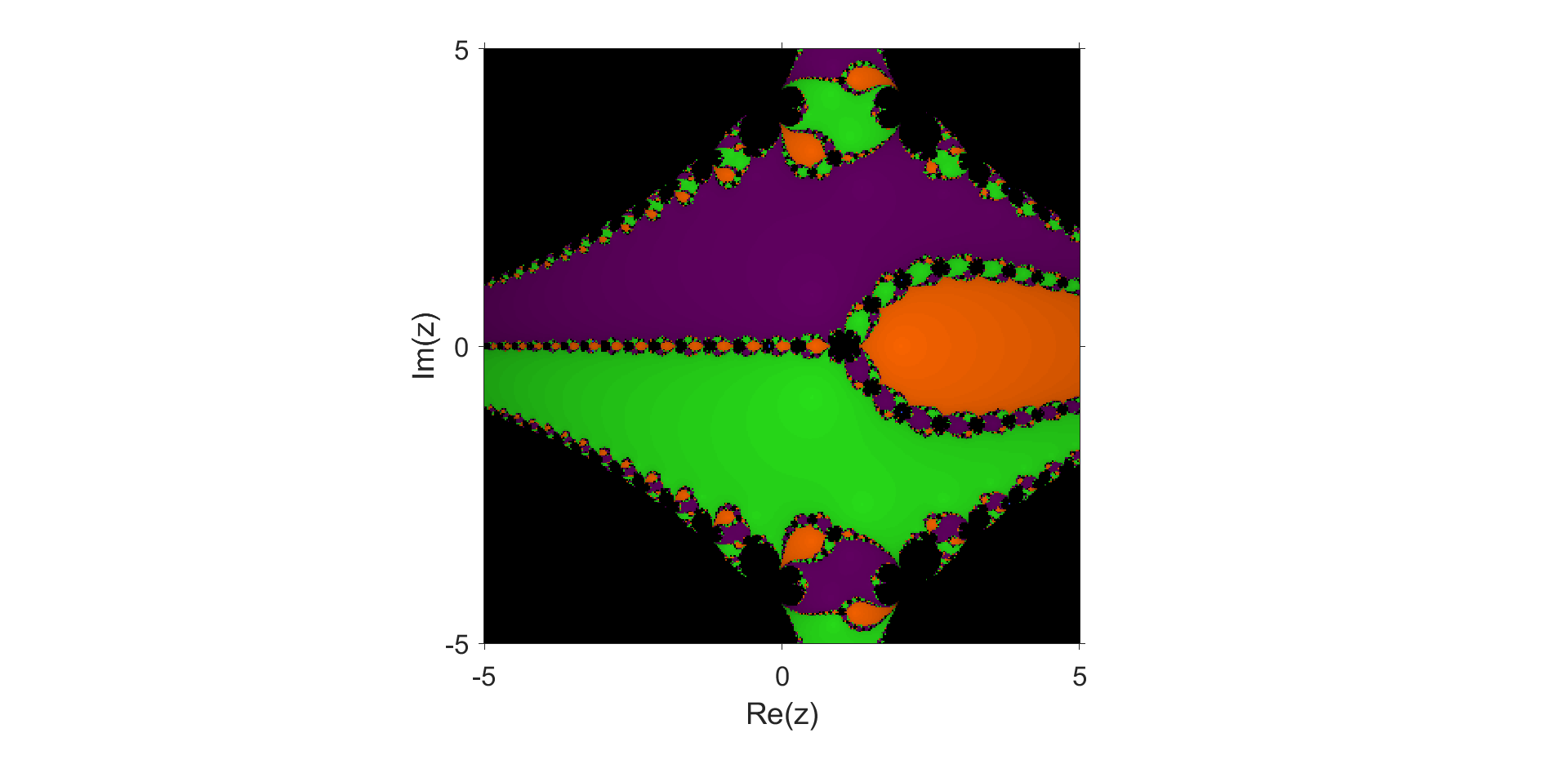}
     \end{subfigure}
     \hfill
     \begin{subfigure}{0.45\textwidth}
     \centering
     \caption{Dynamical plane of $SW_1$ for $p_1$ with $H(t)=1/t$}
     \label{fig:Stef1/t}
     \includegraphics[width=1.5\textwidth]{FIGURES/stefff_b=0.1.png}
     \end{subfigure}
 \end{figure}
 \begin{figure}[H]
     \centering
      \caption{Dynamical planes of $SW_2$ with different weight functions for $p_1$}
     \label{fig2}
     \begin{subfigure}{0.45\textwidth}
     \centering
     \caption{Dynamical plane of $SW_2$ for $p_1$ with $H(t)=1-(t-1)+(t-1)^2$}
     \label{fig:2pasos2}
     \includegraphics[width=1.5\textwidth]{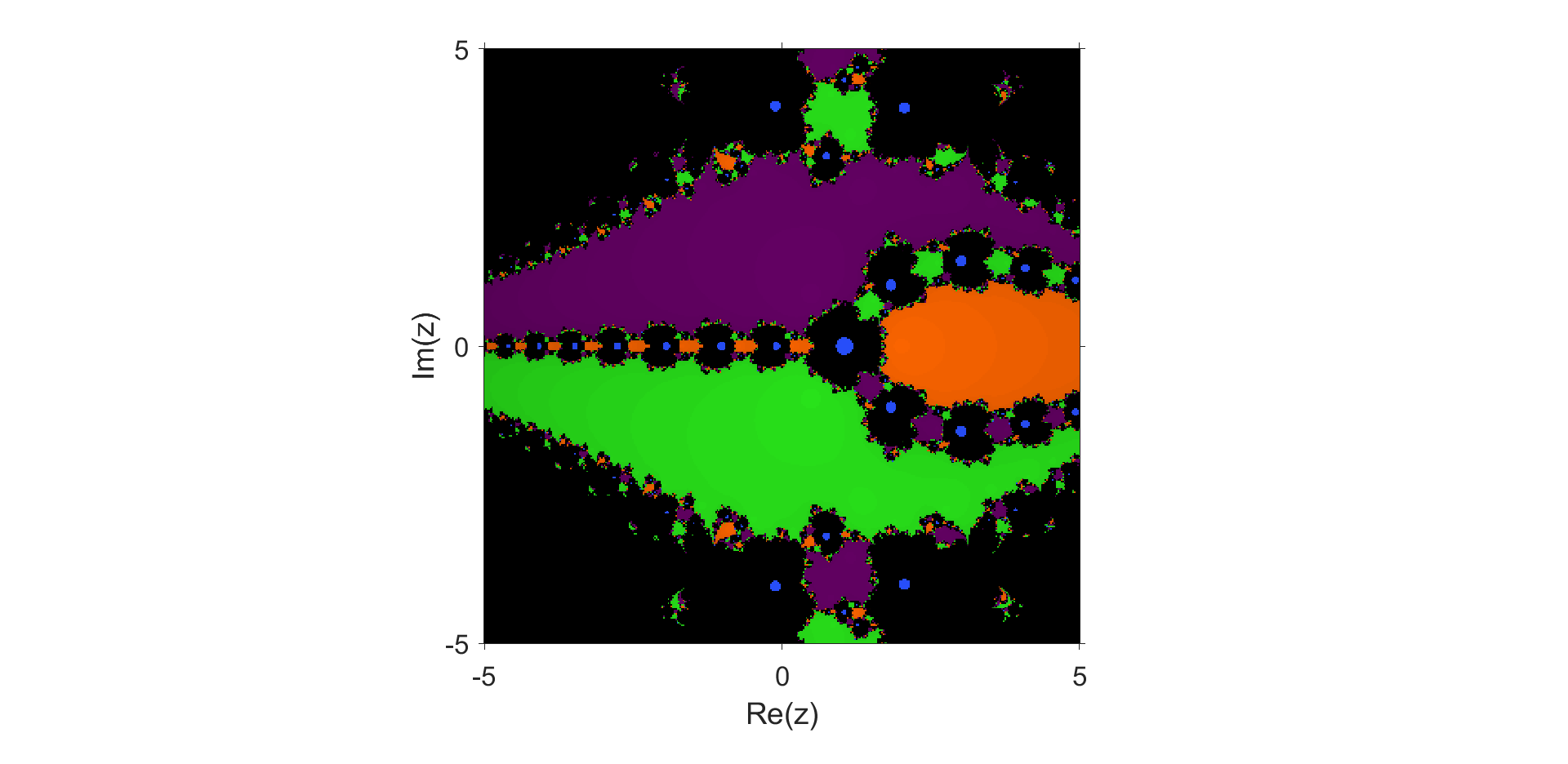}
     \end{subfigure}
     \hfill
     \begin{subfigure}{0.45\textwidth}
     \centering
     \caption{Dynamical plane of $SW_2$ for $p_1$ with $H(t)=1/t$}
     \label{fig:2pasos1/t}
     \includegraphics[width=1.5\textwidth]{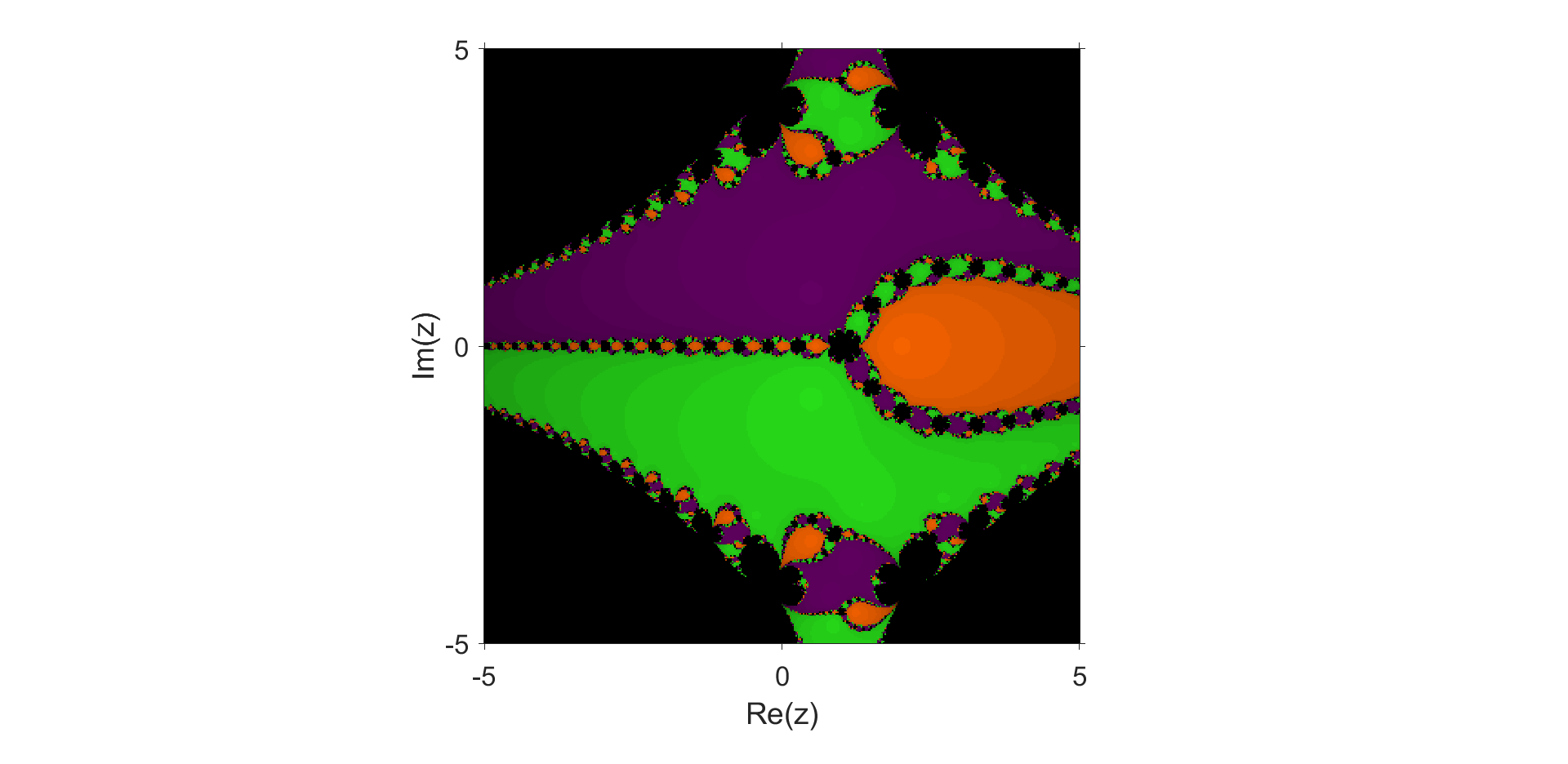}
     \end{subfigure}
 \end{figure}
 \begin{figure}[H]
 \centering
     \caption{Dynamical planes of $SW_3$ with different weight functions for $p_1$}
     \label{fig3}
     \begin{subfigure}{0.45\textwidth}
     \centering
     \caption{Dynamical plane of $SW_3$ for $p_1$ with $H(t)=1-(t-1)+(t-1)^2$}
     \label{fig:3pasos2}
     \includegraphics[width=1.5\textwidth]{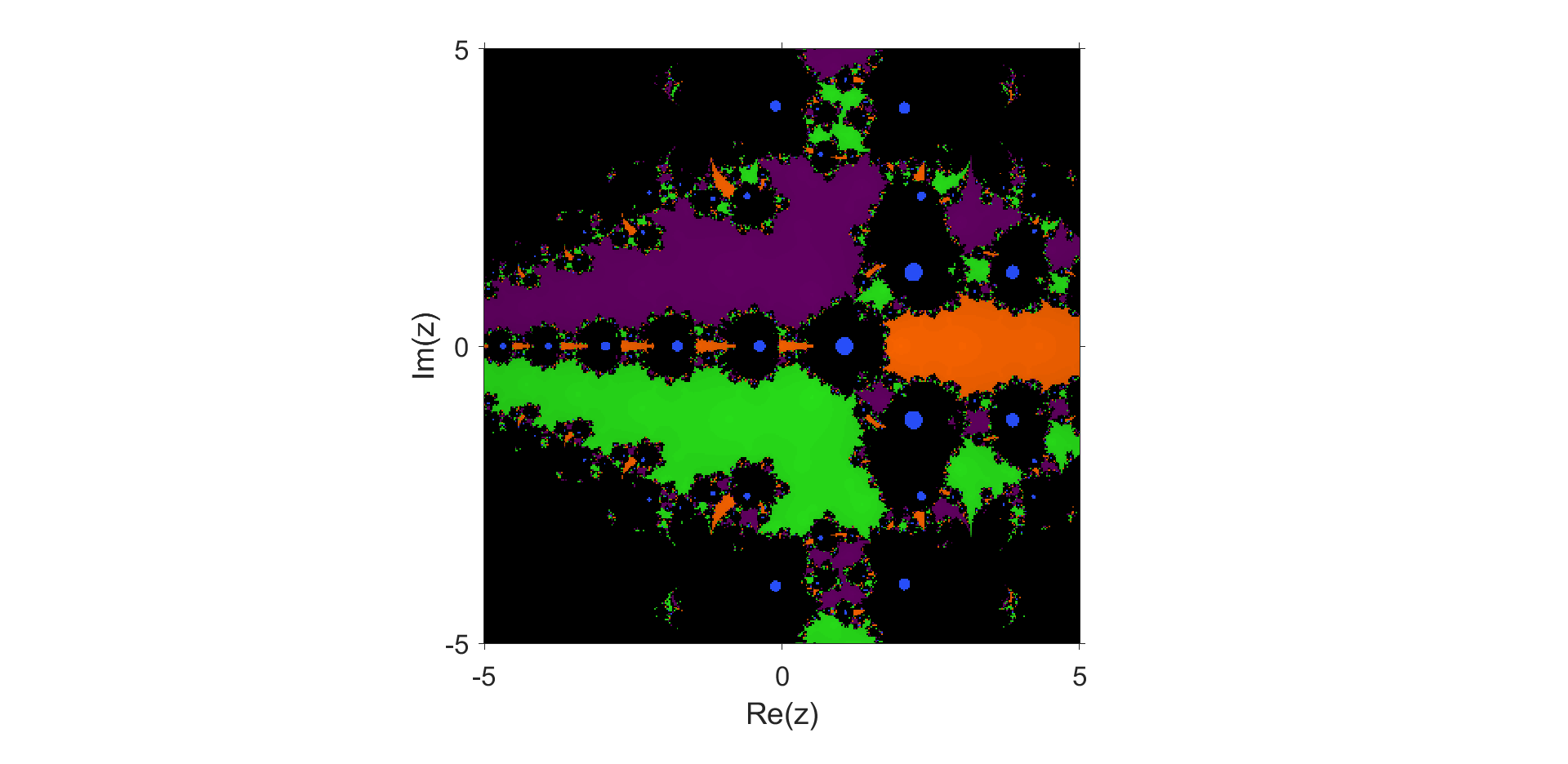}
     \end{subfigure}
     \hfill
     \begin{subfigure}{0.45\textwidth}
     \centering
     \caption{Dynamical plane of $SW_3$ for $p_1$ with $H(t)=1/t$}
     \label{fig:3pasos1/t}
     \includegraphics[width=1.5\textwidth]{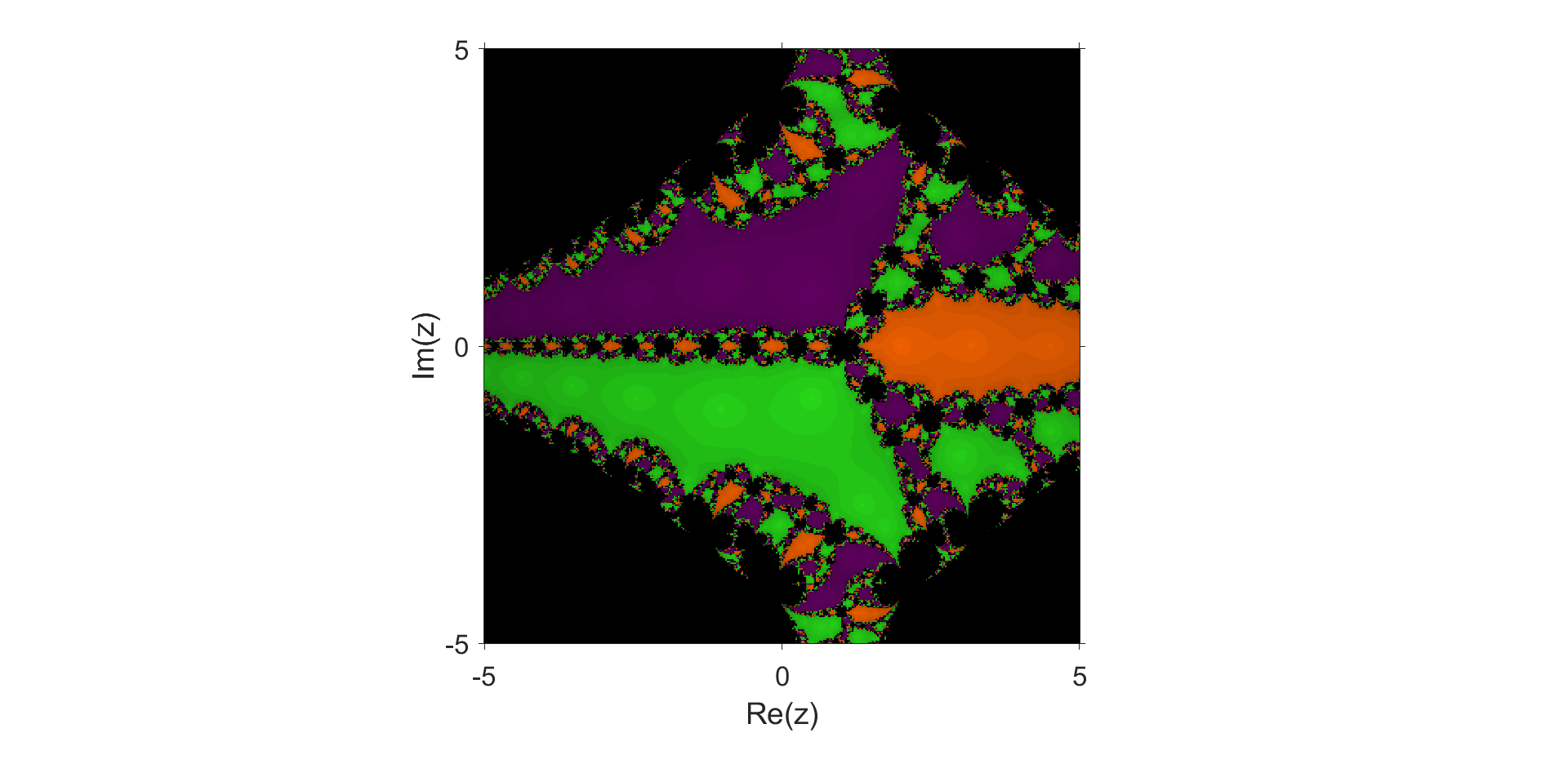}
     \end{subfigure}
 \end{figure}
As we can see in the previous dynamical planes (Figures \ref{fig1}, \ref{fig2}, \ref{fig3}), the convergence zones of all the roots increase when we use the weight function $1/t$ instead of the weight function which is a polynomial. For this reason, for this problem it would be advisable to use the $1/t$ weight function instead of the other, and as we can also see, as we increase the number of steps, the convergence zones decrease, so it would be advisable to use the 1 or 2 step method.

In the case of the methods with memory, we generate the dynamical planes as follows. In this case we set the weight function $H(t)=1/t$.

 To generate the dynamical planes, we denote one of the axis the current iteration and the other axis the previous iteration. That set of points is considered as the initial set of points for the iterative method. Each set of points is painted in a different colour depending on the point to which it converges, where is determined that the inicial point converges to one of the solutions if the distance of the iterations to that solution is less than $10^{-3}$.

These dynamical planes have been generated with a grid of $400\times400$ points and a maximum of $500$ iterations per point to see the difference between the methods better. We paint in orange the initial points that converge to the root $2$, in blue the points that tends to infinity, that have been determined as the points whose absolute value is grater that $10^{150}$, and in black the points that do not converge in less than $500$ iterations.
 \begin{figure}[H]
     \centering
     \caption{Dynamical planes of $SWD_1$ and $SWK_1$ for $p_1$}
     \label{fig4}
     \begin{subfigure}{0.45\textwidth}
     \centering
     \caption{Dynamical plane of $SWD_1$ for $p_1$}
     \label{fig:Stem1}
     \includegraphics[width=1.5\textwidth]{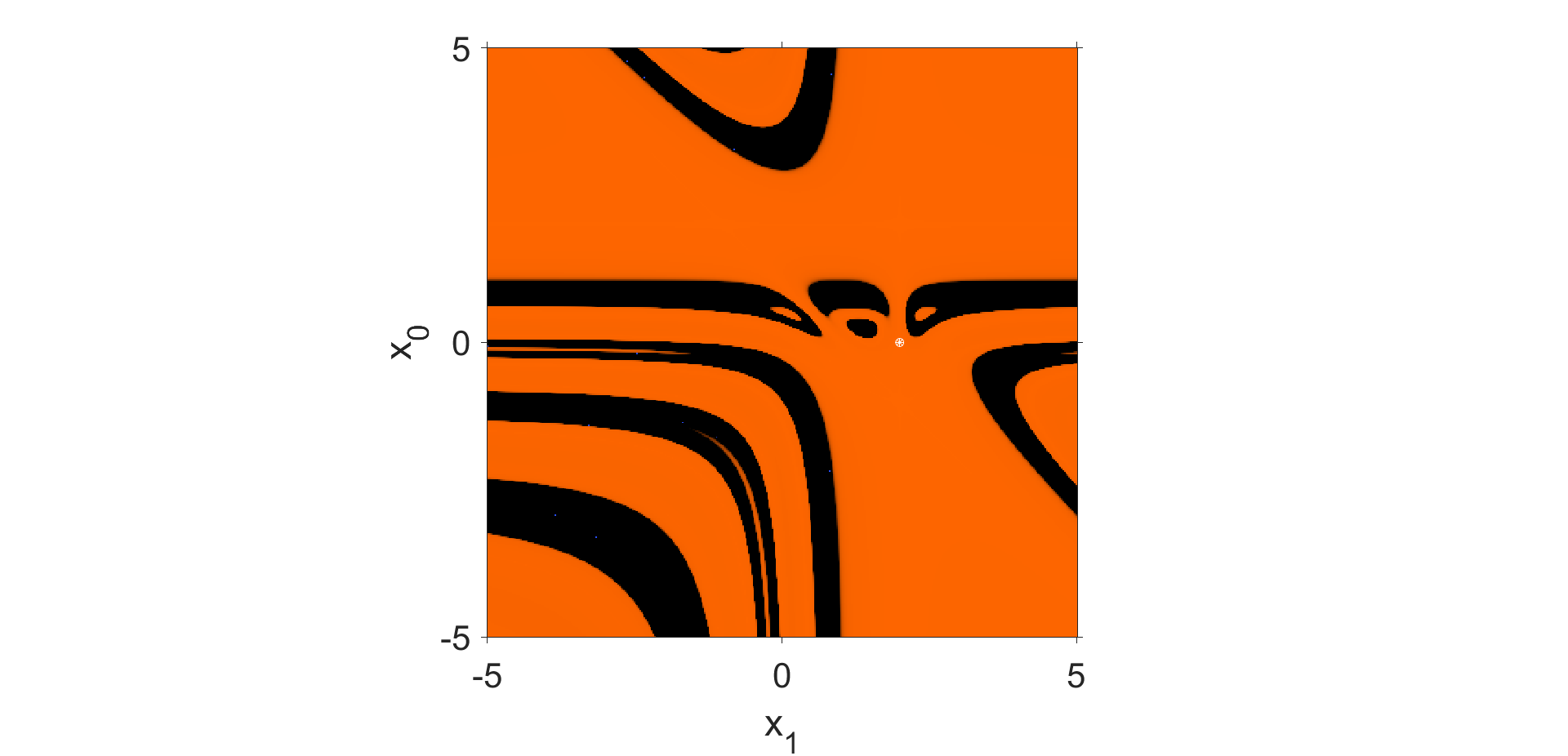}
     \end{subfigure}
     \hfill
     \begin{subfigure}{0.45\textwidth}
     \centering
     \caption{Dynamical plane of $SWK_1$ for $p_1$}
     \label{fig:Stem2}
     \includegraphics[width=1.5\textwidth]{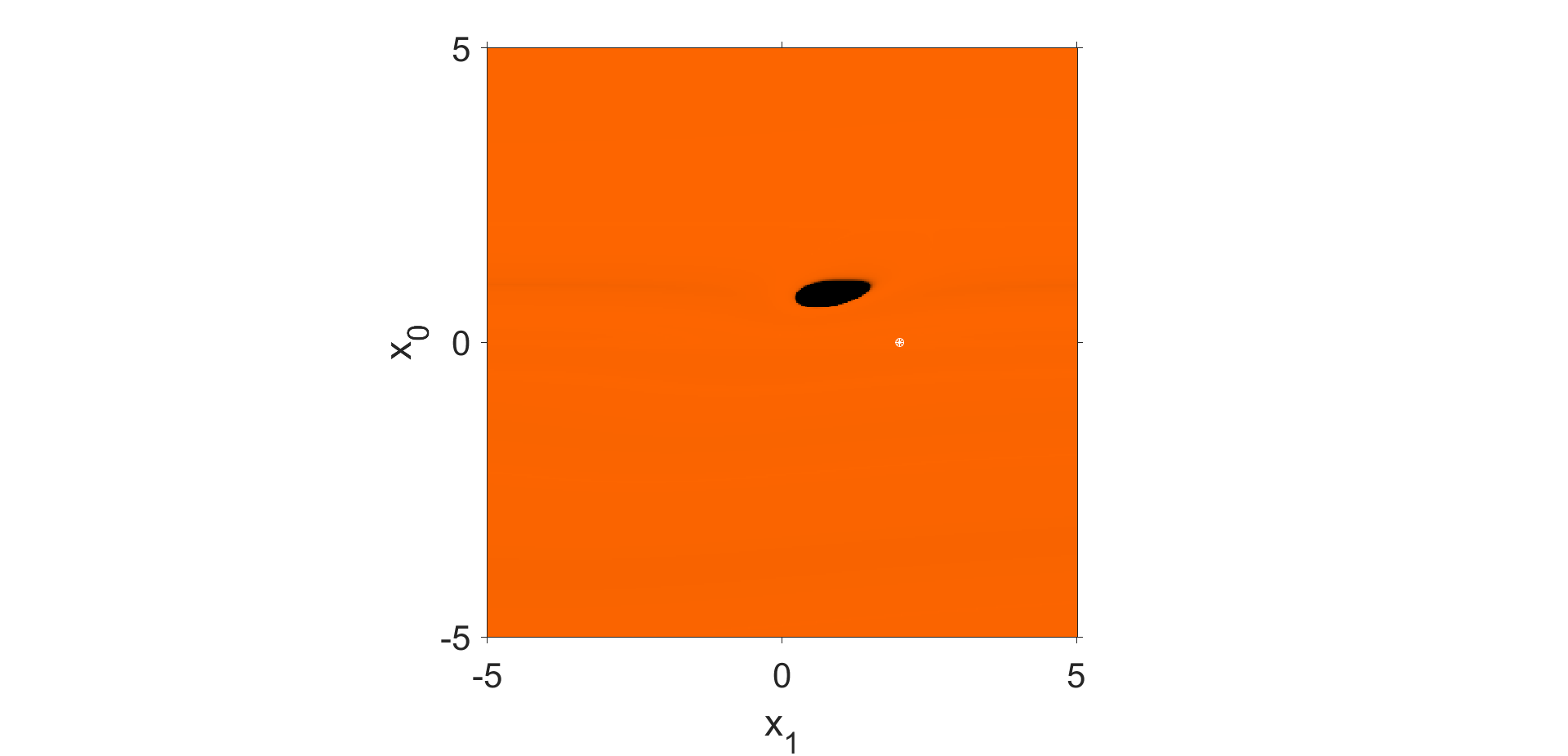}
     \end{subfigure}
 \end{figure}
 \begin{figure}[H]
     \centering
     \caption{Dynamical planes of $SWD_2$ and $SWK_2$ for $p_1$}\label{fig5}
     \begin{subfigure}{0.45\textwidth}
     \centering
     \caption{Dynamical plane of $SWD_2$ for $p_1$}
     \label{fig:2 pasosm1}
     \includegraphics[width=1.5\textwidth]{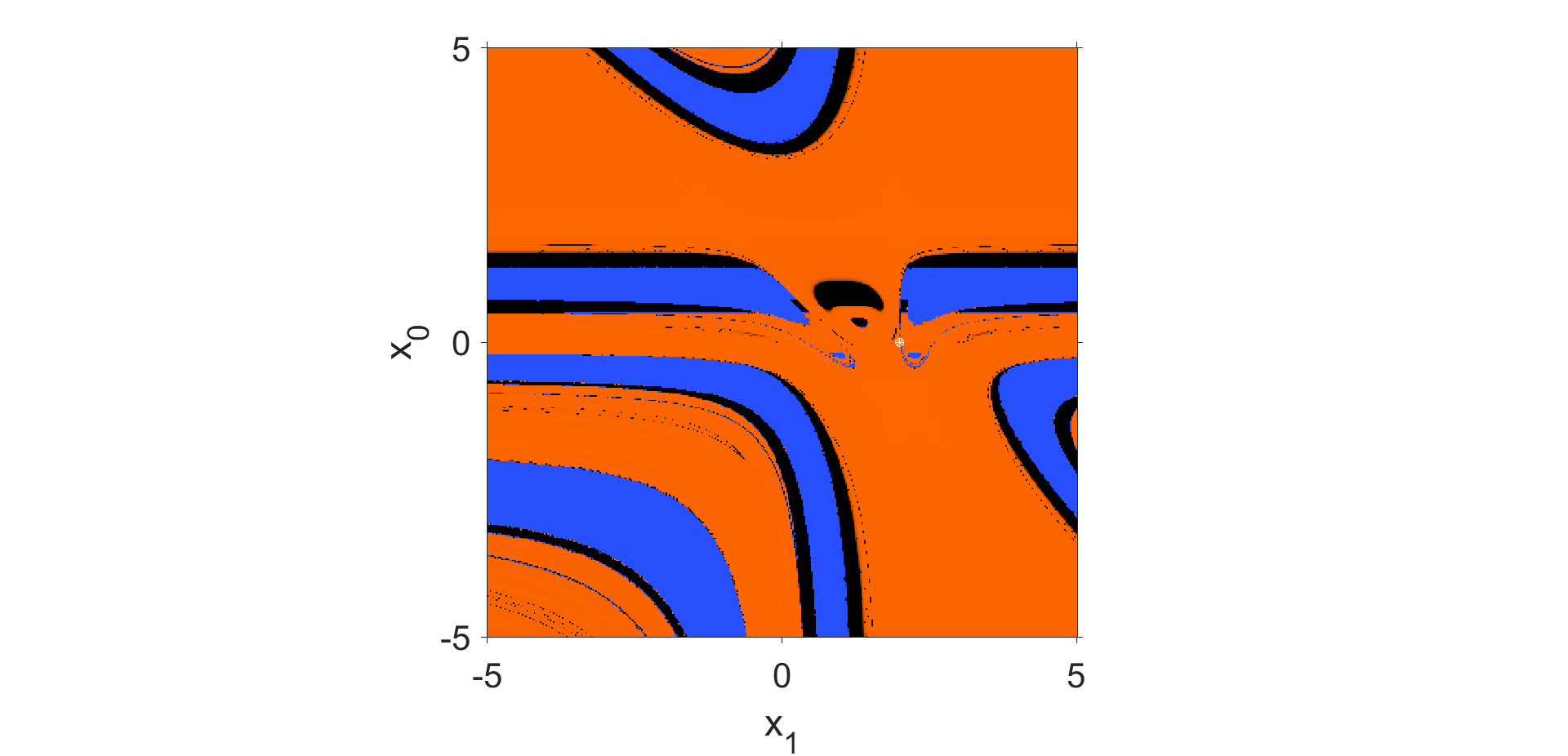}
     \end{subfigure}
     \hfill
     \begin{subfigure}{0.45\textwidth}
     \centering
     \caption{Dynamical plane of $SWK_2$ for $p_1$}
     \label{fig:2 pasosm2}
     \includegraphics[width=1.5\textwidth]{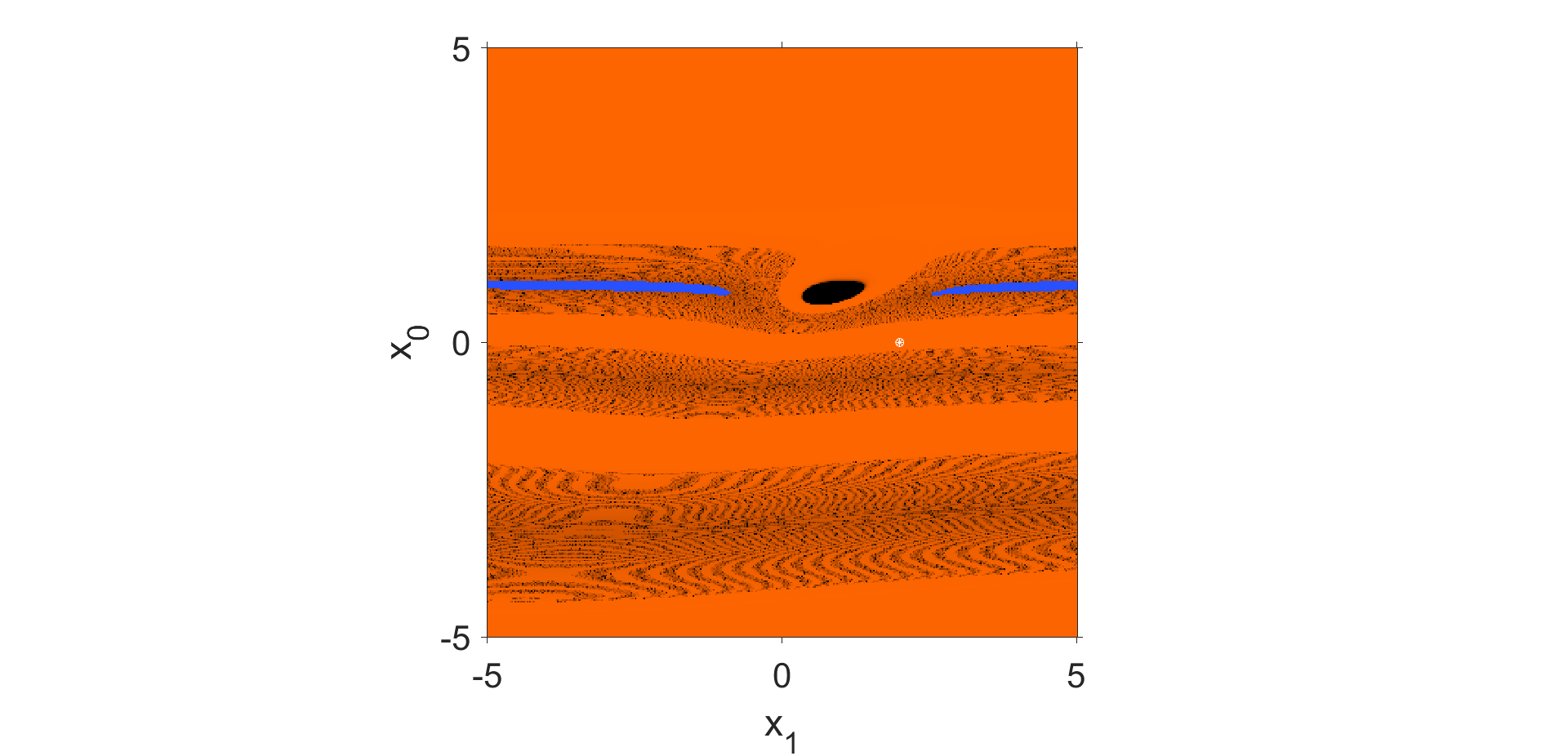}
     \end{subfigure}
 \end{figure}
 \begin{figure}[H]
     \centering
     \caption{Dynamical planes of $SWD_3$ and $SWK_3$ for $p_1$}\label{fig6}
     \begin{subfigure}{0.45\textwidth}
     \centering
     \caption{Dynamical plane of $SWD_3$ for $p_1$}
     \label{fig:3 pasosm1}
     \includegraphics[width=1.5\textwidth]{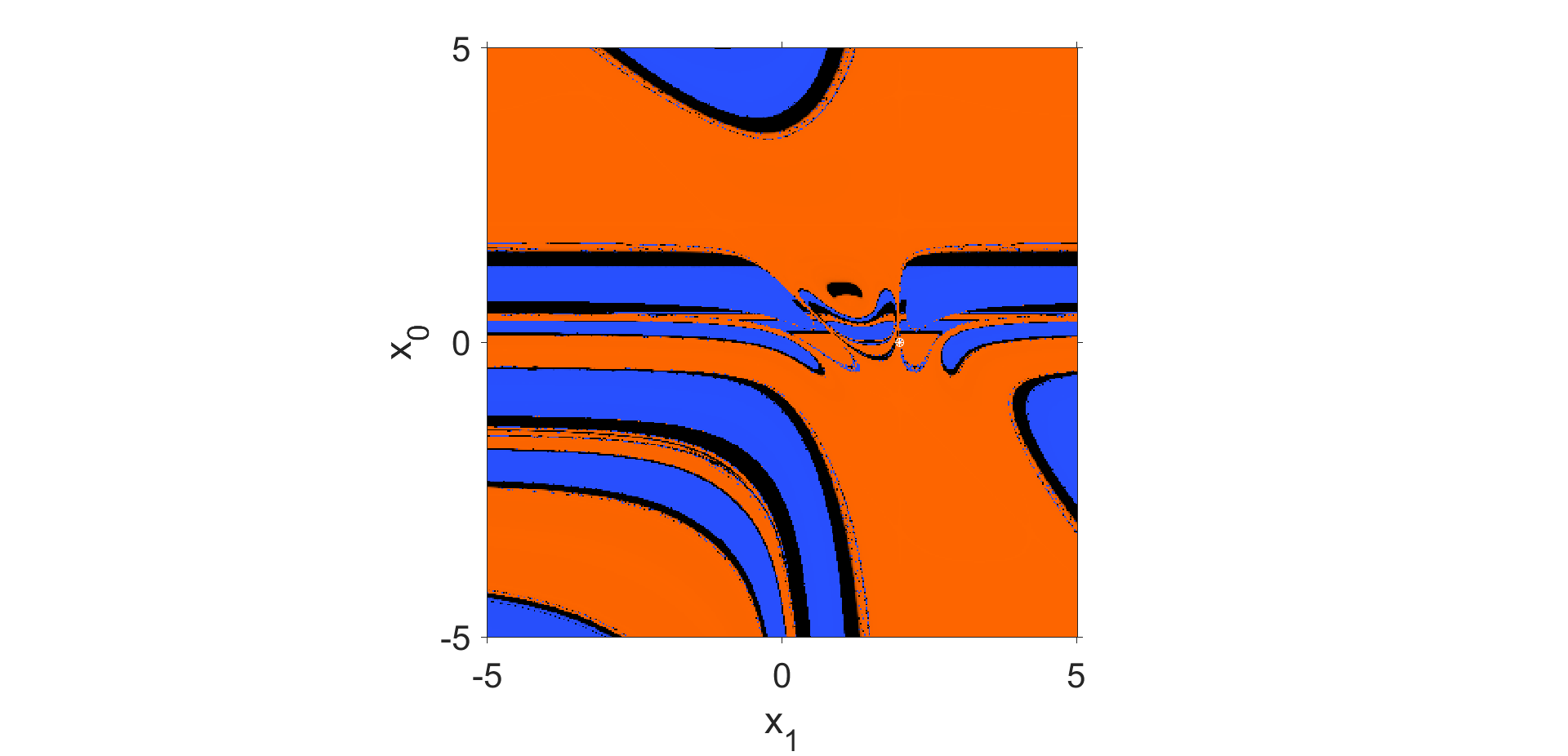}
     \end{subfigure}
     \hfill
     \begin{subfigure}{0.45\textwidth}
     \centering
     \caption{Dynamical plane of $SWK_3$ for $p_1$}
     \label{fig:3 pasosm2}
     \includegraphics[width=1.5\textwidth]{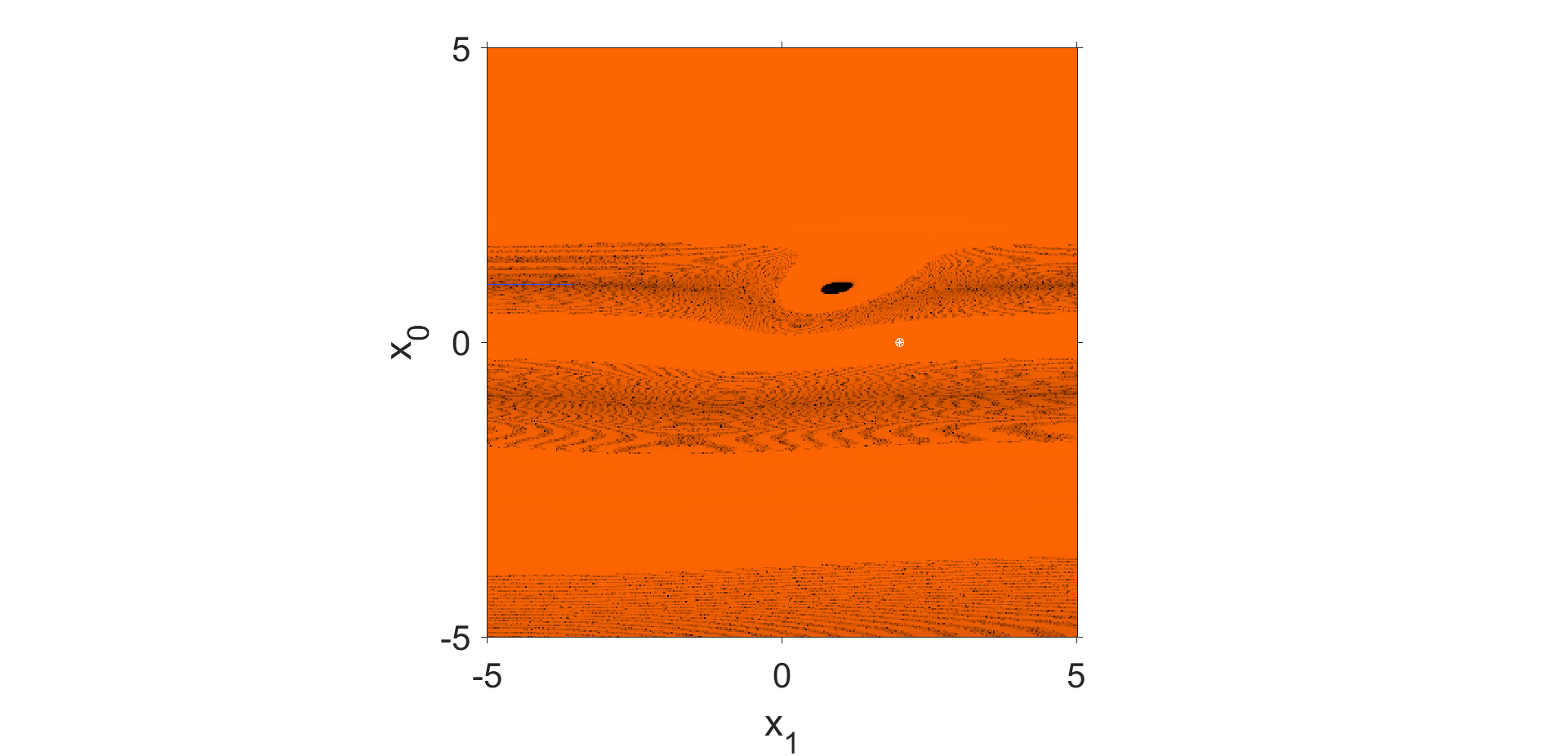}
     \end{subfigure}
 \end{figure}
We note that in the case of the memory methods, as shown in Figures \ref{fig4}, \ref{fig5} and \ref{fig6}, larger convergence areas to the real root are obtained in the case where the Kurchatov divided difference operator is used for any number of steps. For this reason, it would be better to use any of the methods with memory using the Kurchatov's divided difference operator, rather than others.

We discuss the case of methods without memory for solving the non-linear system $p_2$. For all methods the following matrix function has been selected as weight function:
\begin{itemize}
    \item $H(t)=t^{-1}$.
\end{itemize}
To generate the dynamical planes, we have chosen a mesh of $400 \times 400$ points, and what we do is apply our methods to each of these points, taking the point as the initial estimate. Each axis represent each component of the initial point. We have also defined that the maximum number of iterations that each initial estimate must do is $80$, and that we determine that the initial point converges to one of the solutions if the distance to that solution is less than $10^{-3}$. We paint in orange the initial points that converge to the root $(1,1)^T$, in green the initial points that converge to the root $(1,-1)^T$, in blue the initial points that converge to the root $(-1,1)^T$, in red the initial points that converge to the root $(-1,-1)^T$ and in black the initial points that do not converge to any root.
 \begin{figure}[H]
     \centering
     \caption{Dynamical plane of $SW_1$ for $p_2$}
     \label{fig:si1}
     \includegraphics[width=0.8\textwidth]{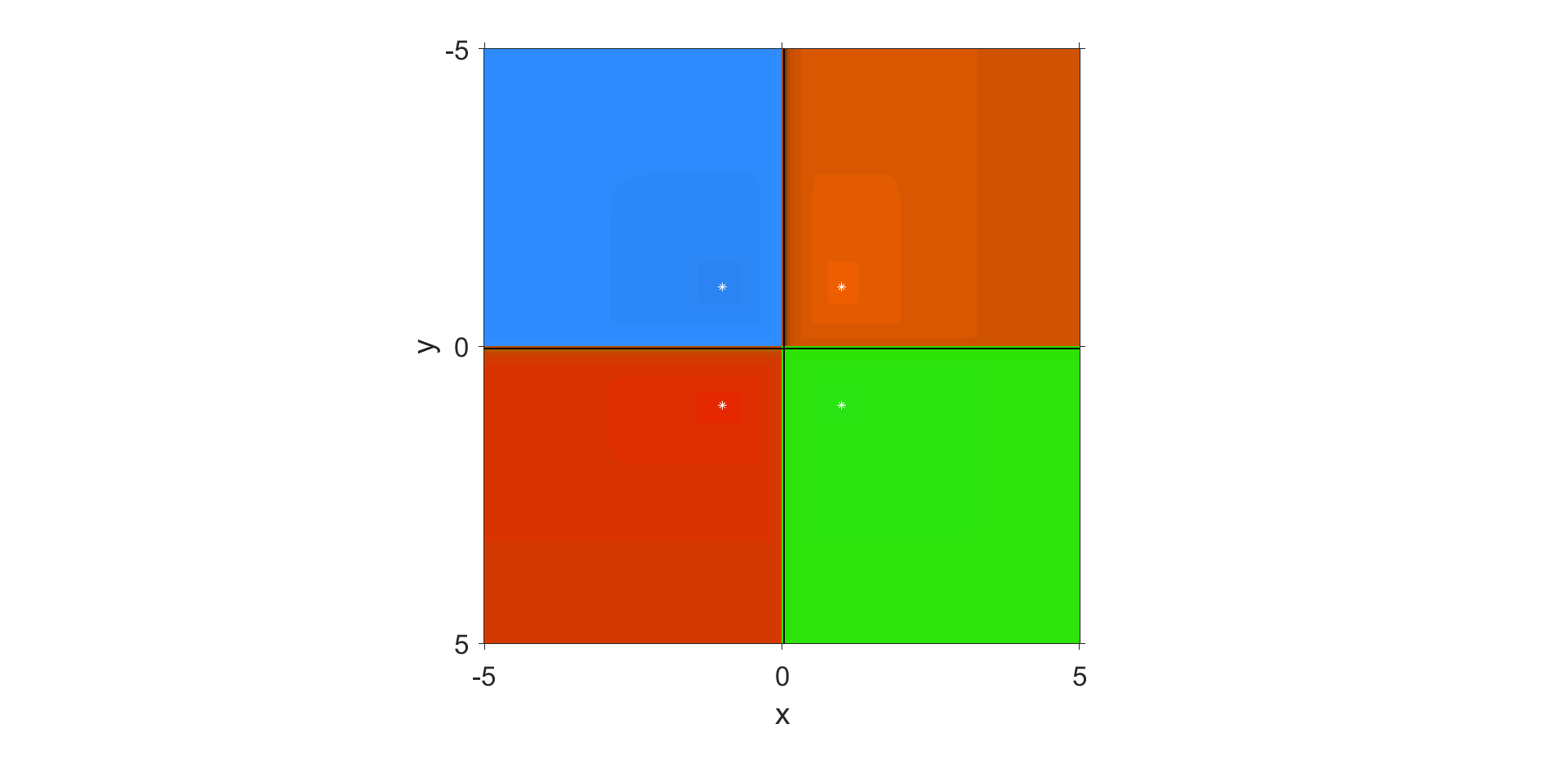}
     \end{figure}
     \begin{figure}[H]
     \centering
     \caption{Dynamical plane of $SW_2$ for $p_2$}
     \label{fig:si2}
     \includegraphics[width=0.8\textwidth]{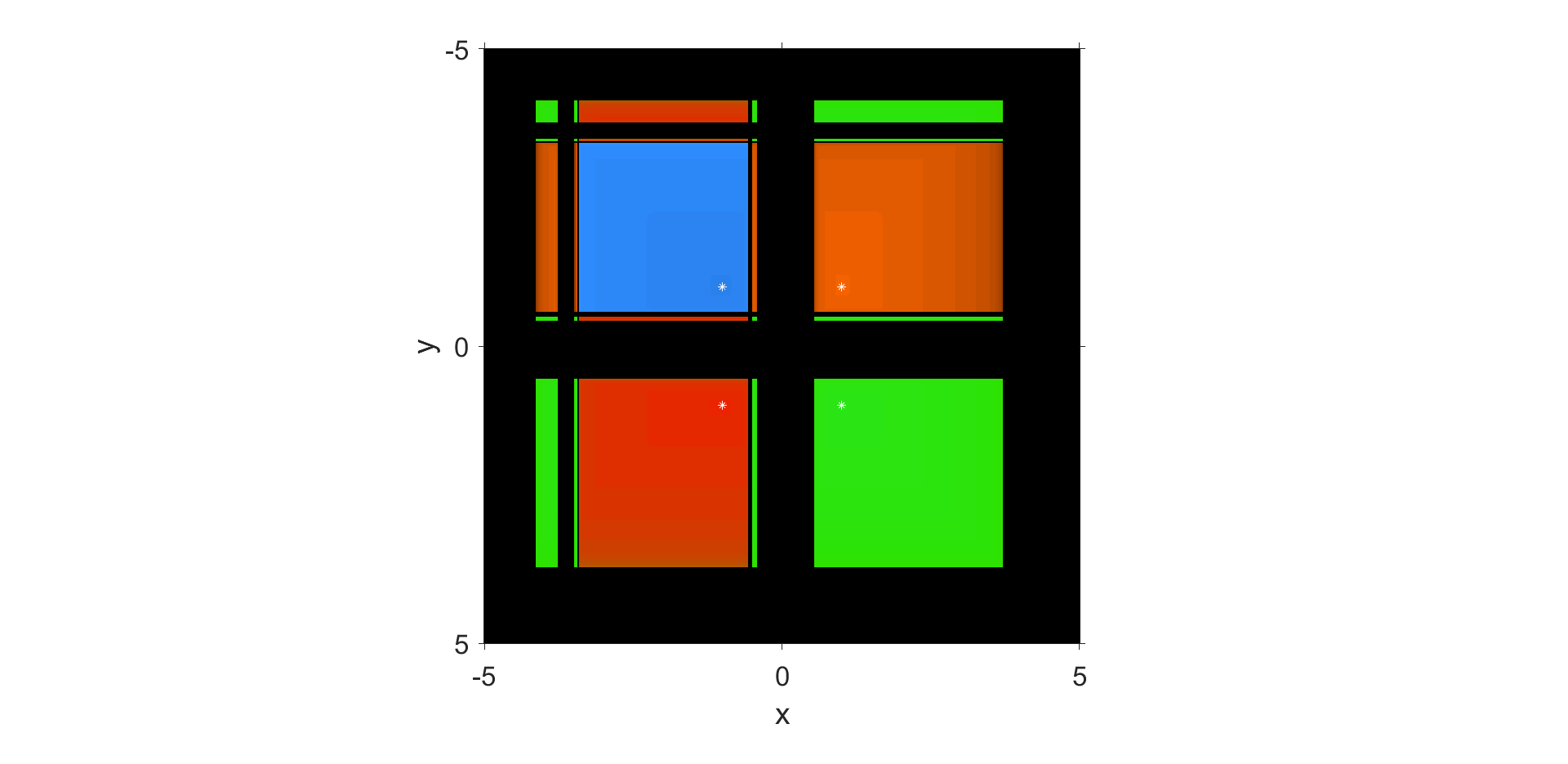}
     \end{figure}
     \begin{figure}[H]
     \centering
     \caption{Dynamical plane of $SW_3$ for $p_2$}
     \label{fig:si3}
     \includegraphics[width=0.8\textwidth]{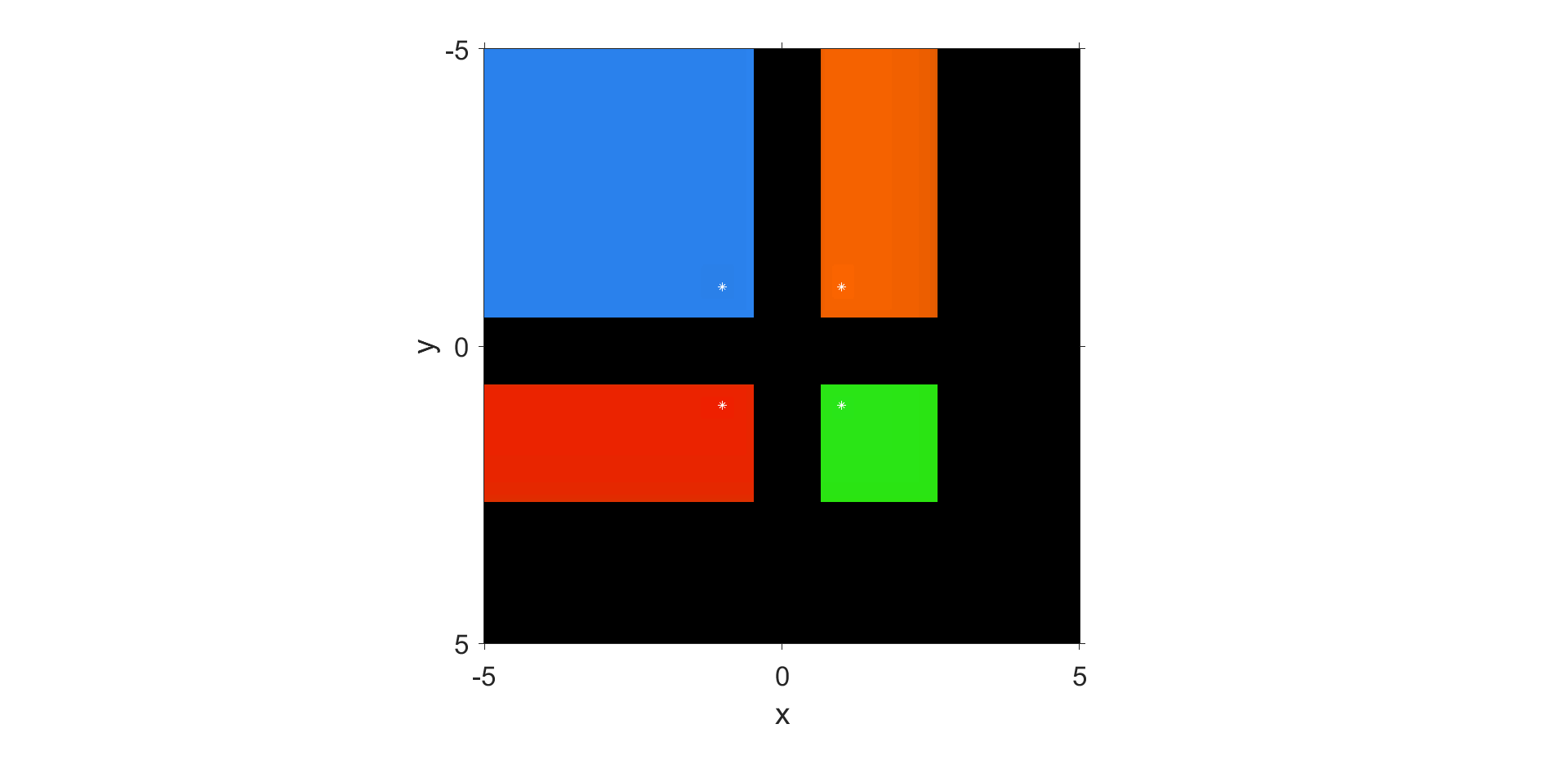}
 \end{figure}
 As we can see in Figures \ref{fig:si1}, \ref{fig:si2} and \ref{fig:si3}, the iterative method that has larger convergence zones to the roots is the $SW_1$ method, so in this case it would be more recommendable to use this method.

\section{Conclusions}
In this manuscript, a class of m-steps derivative-free iterative methods for solving nonlinear systems, has been designed. It has order of convergence $2m$. The expression of the error equation allows us to introduce memory in this family, achieving order $m+\sqrt{m^2+4m+4}$, for $m\geq 2$. We have analyzed the efficiency of the proposed schemes, studying the optimal number of steps providing the best efficiency index.

\section{Author contributions}
All the authors have contributed in the same way in this manuscript.

\section{Financial disclosure}
This research was partially supported by  Universitat Politècnica de
València Contrato Predoctoral PAID-01-20-17 (UPV).

\end{document}